\documentclass{article}
\usepackage{fullpage}

\usepackage{graphicx}  
\usepackage{amssymb}
\usepackage{amsmath}  
\usepackage{amstext}  
\usepackage{amsthm}   
\usepackage{natbib}
\usepackage[flushleft]{threeparttable}  
\usepackage{float}  
\usepackage{hyperref}  
\newtheorem{definition}{Definition}
\usepackage{algorithm,algorithmic}
\newtheorem{theorem}{Theorem}

\newtheorem{lemma}{Lemma}
\newtheorem{corollary}{Corollary}

\hypersetup{colorlinks = true, citecolor = blue}  
\usepackage{subfigure}
\usepackage{xcolor}
\usepackage{booktabs} 
\usepackage{caption}

 \newtheorem{assumption}{Assumption}[section]
\def \E {\mathbb{E}}
\def \X {\mathcal{X}}
\def \dom {\text{dom}}
\def \intdom {\text{int}}
\def \R {\mathbb{R}}
\def \brho {\bar{\rho}}
\def \hx {\hat{x}}
\def \sumT {\sum_{t=1}^T}
\def \fbrho {f_{1/\brho}}
\def \nab {\nabla}
\newcommand{\Et}[1]{\E_t\left[#1\right]}
\newcommand{\Exp}[1]{\exp\left\{#1\right\}}
\newcommand{\norm}[2][]{\left\|#2\right\|^{#1}}
\newcommand{\innerdot}[2]{\left\langle #1 , #2 \right\rangle}
\newcommand{\Brace}[1]{\left(#1\right)}

\usepackage{authblk}
\begin{document}
\title{\bf Stochastic Weakly Convex Optimization Under Heavy-Tailed Noises}
\author[1,2]{Tianxi Zhu\thanks{zhutianxi3291@mail.dlut.edu.cn}}
\author[1]{Yi Xu\thanks{yxu@dlut.edu.cn (corresponding author)}}
\author[3]{Xiangyang Ji\thanks{xyji@tsinghua.edu.cn}}
\affil[1]{School of Control Science and Engineering, Dalian University of Technology}
\affil[2]{School of Computer Science and Technology, Dalian University of Technology}
\affil[3]{Department of Automation, Tsinghua University}
\date{}
\maketitle
\vspace*{-0.5in}
\begin{center}
   First version: July 17, 2025 
\end{center}

\begin{abstract}
An increasing number of studies have focused on stochastic first-order methods (SFOMs) under heavy-tailed gradient noises, which have been observed in the training of practical deep learning models. In this paper, we focus on two types of gradient noises: one is sub-Weibull noise, and the other is noise under the assumption that it has a bounded \( p \)-th central moment ($p$-BCM) with \( p \in (1, 2] \). The latter is more challenging due to the occurrence of infinite variance when \( p \in (1, 2) \). Under these two gradient noise assumptions, the in-expectation and high-probability convergence of SFOMs have been extensively studied in the contexts of convex optimization and standard smooth optimization. However, for \textit{weakly convex} objectives—a class that includes all Lipschitz-continuous convex objectives and smooth objectives—our understanding of the in-expectation and high-probability convergence of SFOMs under these two types of noises remains incomplete. We investigate the high-probability convergence of the vanilla stochastic subgradient descent (SsGD) method under sub-Weibull noises, as well as the high-probability and in-expectation convergence of clipped SsGD under the $p$-BCM noises. Both analyses are conducted in the context of weakly convex optimization. For weakly convex objectives that may be non-convex and non-smooth, our results demonstrate that the theoretical dependence of vanilla SsGD on the failure probability and number of iterations under sub-Weibull noises does not degrade compared to the case of smooth objectives. Under \( p \)-BCM noises, our findings indicate that the non-smoothness and non-convexity of weakly convex objectives do not impact the theoretical dependence of clipped SGD on the failure probability relative to the smooth case; however, the sample complexity we derived is worse than a well-known lower bound for smooth optimization. 
\end{abstract}

\section{Introduction}
In this paper, we focus on the following stochastic optimization problem, which commonly arises in the field of machine learning.
\begin{align}\label{main_problem}
        \min_{x\in\mathcal{X}} f(x),\quad f(x)\triangleq \E_\zeta[f(x,\zeta)]
\end{align}
where $\zeta$ is a random variable following an unknown distribution and $f(x)$ is a closed proper function defined on a closed convex set $\X$. For solving problem~\eqref{main_problem}, stochastic first-order optimization methods (SFOMs) such as stochastic gradient descent (SGD)~\citep{Robbins1951ASA} have been commonly adopted. Unlike deterministic gradient-based methods, SFOMS are done using unbiased or biased estimators of the true gradient of $f$ to improve training efficiency. The convergence of SFOMs, whether in expectation or with high probability, has been extensively studied in the contexts of convex or smooth optimization~\citep{ghadimi2013stochastic,fang2018spider,cutkosky2019momentum,lan2020first, cutkosky2020momentum,liuzijian_2023}. Specifically, we focus on \textit{weakly convex optimization} in problem~\eqref{main_problem}, where $f(x)$ is weakly convex. The definition of weakly convexity is straightforward, i.e., a function $f$ is $\rho$-weakly convex if and only if $f(x)+\rho/2\|x\|^2$ is convex, where $\rho\geq0$. Readers might consider this merely a relaxation of standard convexity. In fact, an important proposition of weakly convex function is that given a $L$-Lipschitz continous and convex function $h:\R^m\mapsto\R$ and a $\beta$-smooth function $c:\R^d\mapsto\R^m$, then their composition $h\circ c$ is a $\beta L$-weakly convex function (see Lemma 4.2 in \citep{drusvyatskiy2017efficiencyminimizingcompositionsconvex}). By letting $h$ and $c$ be the identity function, we can conclude that weakly convex functions include all Lipschitz continuous convex functions and smooth functions, since the identity function is obviously 1-Lipschitz continuous, convex, and 1-smooth. Furthermore, this composition function structure can be connected with the architecture of neural networks. For instance, with a standard loss function like the logistic loss, the sample loss of a neural network employing sigmoid or tanh activation functions exhibits weak convexity while being non-convex and non-smooth.

A lot of prior works has studied the convergence of SFOMs in weakly optimization, such as \citep{Davis2_2018,davis2018stochasticmodelbasedminimizationweakly,Alacaoglu_2020,9345428,li_2022,gao2023delayedstochasticalgorithmsdistributed,WenzhiGao_QiDeng_2024,quanqihu_2024,liao2024errorboundsplcondition,hong_2025}. A key assumption in their theoretical analysis is that the gradient noise is supposed to be variance bounded or sub-Gaussian, which is defined as light-tailed type distribution in statistical terminology. However, recent observations have empirically shown that the distributions of gradient noises may be heavy-tailed, both in training convolutional neural networks~\citep{simsekli2019tail} and BERT~\citep{devlin2019bert}. All these observations motivate us to assume that gradient noises should satisfy weaker assumptions than variance boundedness or sub-Gaussian properties. In this paper, we consider the heavy-tailed gradient noises in weakly convex optimization.

We first focus on the sub-Weibull noises distribution. We say that $X$ follows a $\sigma$-sub-Weibull($\theta$) distribution, i.e., for a random vector $X$, if there exists a nonnegative real number $\sigma$ and satisfy
\begin{align}\label{def:sub_weibull}
        \E\left[\exp\left\{\left(|X|/\sigma\right)^{1/\theta}\right\}\right]\leq 2,
    \end{align}
where $\theta\geq1/2$ is the tail parameter. 
The sub-Weibull distribution is more general than the sub-Gaussian distribution. When $\theta = 1/2$, it reduces to sub-Gaussian distributions, and when $\theta = 1$, it corresponds to sub-exponential distributions. As the tail parameter $\theta$ increases, the tails of sub-Weibull distributions become heavier. For more equivalent definitions of sub-Weibull, we refer interested readers to \citep{vladimirova2020sub}. Prior works such as \citep{madden2020high, lishaojie_liuyong_2022, li2025sharper} have studied the high-probability convergence of SFOMs under sub-Weibull noises in smooth optimization\footnote{Smoothness assumption is stronger than the assumption of weakly convexity}.  Given the existence of weakly convex objectives in machine learning, we pose our first question:

\begin{center}
\textit{Q1: For stochastic weakly convex optimization under sub-Weibull noises, does SGD exhibit a worse dependency on the number of iterations and failure probability \(\delta\) than in smooth optimization cases?}
\end{center}

Note that due to the possible non-differentiability of weakly convex functions, the SGD method here specifically refers to the stochastic sub-gradient descent (SsGD) method.

Next, we focus on the assumption that assume the distribution of gradient noises have bounded $p$-th bounded central moment ($p$-BCM), i.e., suppose that the random vector $X$ is centralized ($\E[X]=0$) and if for $p\in(1,2]$, we have
\begin{align}\label{eq:p-BCM}
        \E[\|X\|^p]\leq \sigma^p\leq +\infty.
    \end{align}
We say random $X$ has bounded $p$-th bounded central moments. The $p$-BCM noise assumption is more challenging than the sub-Weibull noise assumption. Notice that if the gradient noises satisfy $p$-BCM assumption for $p<2$, the second central moment of the noises could be possibly infinite, hence the $p$-BCM assumption is  strictly weaker than the variance bounded assumption. Particularly, the case of $p\in(1,2)$ is much more complicated as the existing theory for vanilla SGD shows that SGD may divergence. A lot of  prior works, such as \citep{zhang_2020,Cutkosky_2021,Nguyen_2023,gorbunov2024highprobabilityconvergencecompositedistributed}, has demonstrated that through gradient clipping—i.e., restricting the stochastic gradient norm to a threshold $\lambda_t$ in $t$-th iteration —some SFOMs can be proven to converge under $p$-BCM gradient noises in smooth or convex optimization. This technique is also widely adopted in practical model training. Naturally, we want to extend the analysis to the weakly convex scenario. Hence, we pose our second interested question:
\begin{center}
\textit{Q2: For stochastic weakly convex optimization under $p$-BCM noises, does clipped-SGD exhibit a worse dependence on the number of iterations and the failure probability $\delta$ than in smooth optimization cases?}
\end{center}

Moreover, our work focuses on high-probability convergence of SFOMs. For a stochastic algorithm $\mathcal{A}$, high-probability convergence emphasizes the upper bounds on the number of iterations or oracle calls needed to find a solution $x$ satisfying $\mathbb{P}\{\mathcal{P}(x)\leq \varepsilon\}\geq1-\delta$ for any $\varepsilon>0$. Here, $\mathcal{P}(x)$ denotes a specific convergence measure of algorithm $\mathcal{A}$ and $\delta$ is commonly called failure probability or confidence level. Compared to in-expectation convergence, high-probability convergence focuses on the probability that a stochastic algorithm exhibits convergence in a single run, rather than its convergence in an average sense. This characteristic makes it more informative for practical applications.

This paper aims to give the convergence analysis of stochastic firs-order methods for solving problem~\eqref{main_problem} under heavy-tailed gradient noises in weakly convex optimization. Our main contributions are summarized as follows. 
\begin{itemize}
    \item We establish the \textbf{first high-probability convergence rate} of SsGD for weakly convex optimization under sub-Weibull noises, which is $  O\big( (\log(T/\delta)^{\min\{0,\theta-1\}} \log(1/\delta)+ \rho \log(1/\delta)^{2\theta} \log T)/\sqrt{T}\big)$ for an unknown time horizon $T$ while $O\big( \sqrt{\left(\rho \log(1/\delta)^{2\theta}\right)/T} + \left(\log(T/\delta)^{\min\{0,\theta-1\}} \log(1/\delta)\right)/T \big)$ for a fixed time horizon $T$. The former result shows that the dependence on the number of iterations and failure probability \( \delta \) aligns with that of the smooth case in the prior work, indicating that under sub-Weibull noises, the theoretical convergence rate of SGD/SsGD methods for non-convex, non-smooth yet weakly convex objectives does not degrade compared to the smooth case.
    
   \item  We establish the \textbf{first high-probability convergence rate} of mini-batch clipped-SsGD for weakly convex optimization under $p$-BCM noises, where the dependence on the failure probability $\delta$ is $O(\log(T/\delta))$ which is consistent with that observed in previous work on smooth optimization, while the sample complexity is $\widetilde{O}(\varepsilon^{-2p/(p-1)})$. When $p<2$ (infinite varance case), it is worse than the well-known lower bound of $\Omega(\varepsilon^{-(3p-2)/(p-1)})$ under $p$-BCM noises for standard smooth optimization.  We also derive in-expectation upper bounds for the mini-batch clipped-SsGD method in weakly convex optimization under $p$-BCM noises. It is worth noting that in our analysis—whether for high-probability convergence or in-expectation convergence—the logarithmic term $\log T$ can be eliminated by fixing the time horizon $T$.
\end{itemize}

\section{Related work}
\subsection{SFOMs under the sub-Weibull noises}
\cite{madden2020high} presented the first high-probability convergence analysis for SGD under sub-Weibull type noises in smooth optimization. They showed that the average of the squared gradient norms of SGD converges to zero with a high-probability upper bound of $O\left(\log T\log(1/\delta)^{2\theta}+\log(T/\delta)^{\min\{0,\theta-1\}}\log(1/\delta)/\sqrt{T}\right)$. When $\theta = 1/2$, their result reduces to the sub-Gaussian case. \cite{lishaojie_liuyong_2022} also derived the same rate under a relaxed condition on the boundedness of true gradients. They further demonstrated that clipped-SGD under sub-Weibull noises can converge at the rate $O((\log^\theta(T/\delta)\log T+\log^{2\theta+1} (T)\log(T/\delta))/\sqrt{T})$ without requiring the true gradients to be bounded. \cite{liu2022convergence} investigated the in-expectation convergence of AdaGradNorm, an adaptive variant of SGD, under sub-Weibull noises in quasar convex and smooth optimization. \cite{liu2024revisitinglastiterateconvergencestochastic} studied the high-probability convergence of the last iterate of SGD under sub-Weibull noises in convex and smooth optimization. \cite{li2025sharper} analyzed the high-probability convergence of momentum-based SGD under sub-Weibull noises and obtained the same convergence rate as \citep{madden2020high} and \citep{lishaojie_liuyong_2022}. \cite{9758044,yu2025distributed} investigated the convergence of (proximal-)SGD under sub-Weibull noises in online optimization and distributed composite optimization, respectively. In another research direction, \cite{chen2025error} studied the algorithmic stability and generalization bounds for SGD in pairwise learning. 

However, the aforementioned works all studied the convergence of SFOMs in convex or smooth optimization under the sub-Weibull noises. To the best of our knowledge, there is no existing literature analyzing the high-probability convergence of the vanilla SGD method under sub-Weibull noises in non-smooth weakly convex optimization.

\subsection{clipped-SFOMs under the \textit{p}-BCM noises}
\cite{zhang_2020} presented the first in-expectation upper bounds for clipped-SGD under the $p$-BCM noise assumption. Their results show that the high probability sample complexity of clipped-SGD is at most $O\left(\max\{\varepsilon^{-(3p-2)/(p-1)},\varepsilon^{-(3p-1)/(2p+2)}\}\right)$ for finding an $\varepsilon$-stationary point in standard smooth optimization. When the problem is assumed to be $\mu$-strongly convex, the complexity improves to $O((\mu\varepsilon)^{-p/(2p-1)})$. \cite{Cutkosky_2021} proposed a normalized SGD method with clipping and momentum, demonstrating  that its high-probability sample complexity under the $p$-BCM noise assumption can achieve $O(\varepsilon^{-(3p-2)/(p-1)})$ by carefully choosing the threshold $\lambda_t$ and fixing the time horizon $T$ in standard smooth optimization. For smooth, convex optimization, \cite{Nguyen_2023} showed that the projected SGD method with clipping has an improved sample complexity of $O(\varepsilon^{-p/(p-1)})$ with a fixed time horizon $T$. \cite{sadiev2023highprobabilityboundsstochasticoptimization} demonstrated that clipped-SSTM (stochastic similar triangles methods) can achieve acceleration with an improved high-probability sample complexity of $O(\varepsilon^{-(p-1)/p})$, and \cite{gorbunov2024highprobabilityconvergencecompositedistributed} proposed a modified proximal SGD method with clipping for distributed composite optimization. \cite{liu_2023stochasticnonsmoothconvexoptimization} provided the first in-expectation and high-probability convergence analysis of the projected clipped-SGD for non-smooth and convex optimization under $p$-BCM noises. \citep{liu_2023c} introduced the first accelerated clipped-SGD method, achieving an improved high-probability sample complexity of $O(\varepsilon^{-(2p-1)/(p-1)})$ under $p$-BCM noises for non-convex, smooth optimization. \cite{chezhegov2025convergenceclippedsgdconvexl0l1smooth} studied clipped-SGD under $p$-BCM noises for convex, ($L_0,L_1$)-smooth optimization, where ($L_0,L_1$)-smoothness extends standard $L$-smoothness, as first proposed by \cite{zhang2020gradientclippingacceleratestraining}. \cite{li2023highprobabilityanalysisnonconvex} investigated the high-probability convergence of the AdaGradNorm method with clipping under the assumption that stochastic gradients have $p$-th bounded moments (a condition strictly stronger than $p$-BCM noises). \cite{chezhegov2025clippingimprovesadamnormadagradnorm} revealed an interesting finding: for convex and smooth problems with specific initial conditions and hyperparameter settings, the convergence rates of Adam and AdaGrad without clipping cannot exhibit polylogarithmic dependence on the failure probability $\delta$ under $p$-BCM noises.  Recently, works such as \citep{liu_2025nonconvexstochasticoptimizationheavytailed, hubler_2025gradientclippingnormalizationheavy, kornilov2025signoperatorcopingheavytailed, he2025complexitynormalizedstochasticfirstorder} have focused on SFOMs without clipping under $p$-BCM noises. 

For lower bound theory, \cite{zhang_2020} demonstrates that under $p$-BCM noises, the convergence rate of any vanilla SGD method cannot exceed $\Omega(T^{-(p-1)/(3p-2)})$ for standard smooth optimization. Additionally, \cite{zhang_2020} shows that for strongly convex optimization under $p$-BCM noises, the convergence rate of any SFOMs cannot be faster than $\Omega(T^{-2(p-1)/p})$. Meanwhile, \cite{vural2022mirrordescentstrikesagain} establishes that for convex optimization under $p$-BCM noises, the convergence rate of any SFOMs is bounded below by $\Omega(T^{-(p-1)/p})$. To the best of our knowledge, no existing work has established lower bounds on the convergence rate of SFOMs under heavy-tailed noises for weakly convex optimization.

Closely related to our work, \cite{LiuLangqi_2024} developed a first-order online optimization algorithm guaranteeing high-probability convergence under $p$-BCM noises for non-smooth, non-convex online optimization—though their analysis allows non-smooth objective functions, it still requires continuous differentiability. \cite{hu_2025} provided the first in-expectation convergence analysis for clipped-SGD under $p$-BCM noises in weakly convex optimization. However, to the best of our knowledge, no prior works have analyzed the high-probability convergence of clipped-SGD under $p$-BCM noises for weakly convex optimization, particularly when $p < 2$. 

\section{Preliminaries}  
\textbf{Notions:} For any $N > 1$, $[N]$ denotes the set $\{1, \dots, N\}$. Given $x, y \in \mathbb{R}^d$, $\langle x, y \rangle$ represents the standard Euclidean inner product, i.e., $\langle x, y \rangle = x^T y$. In our analysis, $\|\cdot\|$ denotes the $\ell_2$ norm. For a set $S$, $\intdom(S)$ signifies the interior of $S$. $\textrm{Proj}_\mathcal{X}(y)$ means projecting $y$ onto a given closed convex set $\mathcal{X}$. $\partial f(x)$ denotes the sub-differential set of a function $f$ at $x$. $g(x, \zeta)$ denotes the stochastic gradient at $x$ with respect to the random variable $\zeta$.

Since the objective function \( f \)  in Problem~\eqref{main_problem} is potentially non-smooth and non-convex, the existence of the classical convex sub-gradient \( \partial f \) cannot be ensured. In this context, we leverage the Fréchet sub-gradient, which exists if \( f \) is lower semi-continuous at the point of interest. The following definition of Fréchet sub-differential aligns with Definition 1.1 in~\citep{Kruger_Alex_2003}.
\begin{definition}{(Fréchet sub-differential)}
Let $f:\R^d\mapsto\bar{\R}$ be proper and closed. If $f$ is finite at $x$, we can define a set
\begin{align*}
    \partial f(x) \triangleq \left\{g\in\R^d:\lim_{y\rightarrow x}\inf\frac{f(y)-f(x)-\langle g,y-x\rangle}{\|y-x\|}\geq 0\right\},
\end{align*}
which is called a (Fréchet) sub-differential of $f$ at $x$ and each elements of $\partial f(x)$ are refered to as (Fréchet) sub-gradients (regular gradients). 
\end{definition}
Fréchet sub-gradients can be regarded as an extension of convex sub-gradients or gradients. Specifically, if 
$f$ is convex, the Fréchet sub-differential reduces to the classical convex sub-differential, i.e., $\partial f(x) = \{g:f(y)-f(x) \geq \langle g, y-x\rangle, \forall y\in\dom(f)\}$. If $f$ is differentiable at $x$, the Fréchet sub-differential collapses to the singleton set containing only the gradient: $\partial f(x) = \{\nabla f(x)\}$

In the context of weakly convex optimization, defining the \textit{stationary point} of $f$
 is crucial due to its non-smoothness and non-convexity. \citep{Davis_2018} is the first work to use the gradient of the \textit{Moreau envelope} of a weakly convex function to analysis the convergence of the stochastic sub-gradient method. For a $\rho$-weakly convex function $f$, its Moreau envelope is defined as follows:
 \begin{align}~\label{eq: fun_moreau_envelope}
     f_{1/\bar{\rho}}(x) \triangleq \min_{y\in\R^d}\left\{f(y)+\mathbb{I}_\X(y)+\frac{\bar{\rho}}{2}\|y-x\|^2\right\},
 \end{align}
 where $\bar{\rho}>\rho$ and $\mathbb{I}_\X$ is the indicator function defined over $\X$, i.e., $\mathbb{I}_\X(y)=0$, if $y\in\X$, otherwise, $\mathbb{I}_\X(y)=+\infty$. Let we denote $\hat{x}\triangleq\arg\min_{y\in\R^d}\left\{f(y)+\textrm{I}_\X(y)+\frac{\bar{\rho}}{2}\|y-x\|^2\right\}$, we have following properties of $f_{1/\bar{\rho}}$:
 \begin{align*}
     \begin{cases}
         &\nabla f_{1/\bar{\rho}}(x) = \bar{\rho}(x-\hat{x})\\
         &f(\hat{x})\leq f(x)\\
         &dis(0,\partial f(\hat{x}))\leq \|\nabla f_{1/\bar{\rho}}(x)\|
     \end{cases}
 \end{align*}
 From the preceding inequalities, it can be readily inferred that if $\|\nabla f_{1\bar{\rho}}(x)\|\leq \varepsilon$ then there exits a point $\hat{x}$ that is near $x$ such that $dis(0,\partial f(\hat{x}))\leq\varepsilon$. Consequently, we leverage the gradients of the Moreau envelopes of $f$ to quantify the convergence of the algorithm.

 The following standard assumptions are essential for our analysis:
 \begin{assumption}\label{assp:weak_conv}
     $f$ is a $\rho$-weakly convex function, where $\rho>0$.
 \end{assumption}
 \begin{assumption}\label{assp:lower_bound}
     $f(x)$ has at least one local minimum $f^*$ on $\mathcal{X}$ and $f^*> -\infty$;
 \end{assumption}

 \begin{assumption}\label{assp:G_Lipsch}
     $\exists G>0,\, \forall x\in\mathcal{X},\,\forall g\in\partial f(x),\, \|g\|\leq G$;
 \end{assumption}

 \begin{assumption}\label{assp:unbiased}
     The stochastic gradient oracle is unbiased, i.e., $\E_\zeta[g(x,\zeta)]\in\partial f(x)$.
 \end{assumption}

 \begin{assumption}\label{assp:sub_weibull}
      The norm of gradient noise $\|g(x,\zeta)-\E_\zeta[g(x,\zeta)]\|$ is a $\sigma$-sub-Weibull($\theta$) random, where $\theta\geq1/2$ 
 \end{assumption}

\begin{assumption}\label{assp:p_PCM}
    The gradient noise $g(x,\zeta)-\E_\zeta[g(x,\zeta)]$ satisfies $p$-BCM assumption for $p\in(1,2]$.
\end{assumption}
Additionally,we require $\X \subseteq \intdom(\dom f)$ to ensure the existence of sub-gradients of $f$ over $\X$.
Assumption~\ref{assp:lower_bound} and Assumption~\ref{assp:unbiased} are standard conditions in non-smooth stochastic optimization, while Assumption~\ref{assp:G_Lipsch} has been widely adopted in weakly convex optimization\citep{Davis_2018,Davis2_2018,Alacaoglu_2020,li_2022,hu_2025,hong_2025}. 

\section{Main results}
In this section, we formally present the main results of this paper. First, we analyze the high-probability convergence of the projection stochastic sub-gradient descent (SsGD) method with sub-Weibull-type gradient noise under the weakly convex setting in Section~\ref{sec:SsGD_subweibull}. Furthermore, still within the weakly convex framework, we study the high-probability convergence of Clipped-SsGD with $p$-th bounded noise, which allows the second-moment norm of gradient noise to be infinite when $1 < p < 2$ in Section~\ref{sec:Cipped-SsGD_p_BCM}. It is worth noting that none of our analyses require the constraint domain $\mathcal{X}$ to be compact.
\subsection{SsGD under the sub-Weibull noises}\label{sec:SsGD_subweibull}
First, we present the projected SsGD algorithm as follows.
\begin{algorithm*}[t]
    \caption{Projected SsGD method} \label{algo:Projected_SsGD}
    \begin{algorithmic}[1]
       \STATE \textbf{Input}: $x_0 \in \X$ and $\eta_0>0$
       \FOR{$t=1,\cdots,T$}
       \STATE Access to unbiased stochastic gradient oracle to get $g_t$
       \STATE $x_{t+1} = \mathrm{Proj}_\X (x_t - \eta_t g_t)$
       \ENDFOR
    \end{algorithmic}
\end{algorithm*}

Now, we summarize a general high-probability convergence theorem for the projected SsGD method with sub-Weibull-type noise under the weakly convex setting. This theorem obviates the need for additional explicit form of $\eta_t$. The detailed proofs are include in the Appendix.

\begin{theorem}\label{thm: SsGD}
    Suppose Assumptions~\ref{assp:weak_conv}, \ref{assp:lower_bound}, \ref{assp:G_Lipsch}, \ref{assp:unbiased} and \ref{assp:sub_weibull} hold. Let $\{x_t\}$ be the iterate produced by projection-SsGD. Let $f_{1/\bar{\rho}(x_t)}$ is defined as~\eqref{eq: fun_moreau_envelope} and set $\bar{\rho}=3\rho$. If $\eta_t$ satisfy 
    \begin{align}
        \sum_{t=1}^T \eta_t = +\infty,\quad \sum_{t=1}^T\eta_t^2 \leq +\infty
    \end{align}
    Then given any $\delta\in(0,1)$, with the probability at least $1-\delta$:
        \begin{itemize}
        \item when $\theta=\frac{1}{2}$, we have the following inequality
            \begin{align}
                \sum_{t=1}^T \frac{\eta_t}{\sum_{s=1}^T\eta_s} \|f_{1/3\rho}(x_t)\|^2 \le O&\left(\frac{\sigma^2\log(1/\delta)\max_{t\in[T]}\eta_t}{\sum_{t=1}^T\eta_t}\right.\nonumber\\
        &\left.+\frac{\Delta_1+\rho(\log(1/\delta)\sigma^2+G^2)\sum_{t=1}^T\eta_t^2}{\sum_{t=1}^T\eta_t}\right)
            \end{align}
            \item when $\theta = (\frac{1}{2},1]$, we have the following inequality:
            \begin{align}
                \sum_{t=1}^T\frac{\eta_t}{\sum_{s=1}^T\eta_s}\|\nabla f_{1/3\rho}(x_t)\|^2\leq O&\left(\frac{\Delta_1+\max\{\sigma^2,\sigma G\}\log(1/\delta)\max_{t\in[T]}\eta_t}{\sum_{t=1}^T\eta_t}\right.\nonumber\\
                &\left.+\rho([\log(1/\delta)]^{2\theta}\sigma^2+G^2)\frac{\sum_{t=1}^T\eta_t^2}{\sum_{t=1}^T\eta_t}\right)
            \end{align}
            \item when $\theta>1$, we have the following inequality
            \begin{align}
                \sum_{t=1}^T\frac{\eta_t}{\sum_{s=1}^T\eta_s}\|\nabla f_{1/3\rho}(x_t)\|^2\leq O&\left(\frac{\Delta_1+D(\theta)\log(1/\delta)\max_{t\in[T]}\eta_t}{\sum_{t=1}^T\eta_t}\right.\nonumber\\
        &\left.+\rho([\theta\log(1/\delta)]^{2\theta}\sigma^2+G^2)\frac{\sum_{t=1}^T\eta_t^2}{\sum_{t=1}^T\eta_t}\right)
            \end{align}
        \end{itemize}
    where $\Delta_1\triangleq f_{1/3\rho}(x_1)-\min_\X f(x)$, $D(\theta)=\max\{ G[\log(T/\delta)]^{\theta-1}\sigma,2^{3\theta+1}\Gamma(3\theta+1)\sigma^2\}$ and $\Gamma(\cdot)$ is denoted as \textit{Gamma function}.
\end{theorem}

In Theorem~\ref{thm: SsGD}, we omit certain constants in the big-O complexity of the upper bound, and the full version of Theorem~\ref{thm: SsGD} can be found in section \ref{sec: SsGD_subweibull}. The requirement of $\bar{\rho}=3\rho$ is only to simplify the upper bound in Theorem~\ref{thm: SsGD}. In fact, our analysis allows for any $\bar{\rho}>\rho$.

From Theorem~\ref{thm: SsGD}, it is straightforward to observe that as $\theta$ increases—indicating a heavier tail in the gradient noise distribution—the iteration complexity of SsGD deteriorates. This is consistent with empirical observations when SsGD handles heavy-tailed gradient noise. Interestingly, when $\theta>1$, although we did not specify the explicit form of $\eta_t$, an extra term $\log(1/\delta)\log(T/\delta)^{\theta-1}$ emerges in the upper bound complexity. In our analysis, the appearance of $\log(1/\delta)\log(T/\delta)^{\theta-1}$ arises because we employ the sub-Weibull Freedman inequality (see Theorem \ref{app_thm:sub_weibull_freedman_ineq}), which can be viewed as an extension of the sub-Gaussian freedman inequality (see Lemma 1 in~\citep{lixiaoyu_2020}). Note that while the highest order of $\sigma$ in Theorem~\ref{thm: SsGD} is $2$ for $\theta > 1$, the constant coefficient preceding $\sigma^2$ is $\Gamma(3\theta + 1)$—defined as the extension of the Factorial function—which grows extremely large as $\theta$ increases. All above analysis implies that SsGD performs poorly with $\sigma$-sub-Weibull($\theta$)-type gradient noise when $\theta > 1$, even though the square of the noise norm remains bounded.

To facilitate a better comparison with previous work, we first set $\eta_t=O(1/\sqrt{t})$. Then the following corollary can be derived from Theorem~\ref{thm: SsGD} easily.

\begin{corollary}\label{coro: time_varying_step_SsGD_subweibull}
Suppose Assumptions~\ref{assp:weak_conv}, \ref{assp:lower_bound}, \ref{assp:G_Lipsch}, \ref{assp:unbiased} and \ref{assp:sub_weibull} hold. Let $\{x_t\}$ be the iterate produced by projection-SsGD. Let $f_{1/\bar{\rho}(x_t)}$ is defined as~\eqref{eq: fun_moreau_envelope} and set $\bar{\rho}=3\rho$. Let $\eta_t=\frac{\gamma}{\sqrt{t}}$, where $\gamma$ can be any positive number.  Given any $\delta\in(0,1)$, the SsGD method can converge with the probability at least $1-\delta$
        \begin{itemize}
            \item when $\theta=\frac{1}{2}$, we have the following inequality:
            \begin{align}
                \frac{1}{T}\sum_{t=1}^T \|\nabla f_{1/3\rho}(x_t)\|^2&\leq O\left(\frac{\Delta_1}{\sqrt{T}}+\rho(\log(1/\delta)\sigma^2+G^2)\frac{\log(T)}{\sqrt{T}}\right)
            \end{align}
            \item when $\theta = (\frac{1}{2},1]$, we have the following inequality:
            \begin{align}
                \frac{1}{T}\sum_{t=1}^T \|\nabla f_{1/3\rho}(x_t)\|^2\leq O\left(\frac{\Delta_1+\sigma G\log(1/\delta)}{\sqrt{T}}+\frac{\rho([\log(1/\delta)]^{2\theta}\sigma^2+G^2)\log T}{\sqrt{T}}\right)
            \end{align}   
            \item when $\theta > 1$, we have the following inequality:
            \begin{align}
                 \frac{1}{T}\sum_{t=1}^T \|\nabla f_{1/3\rho}(x_t)\|^2&\leq O\left(\frac{\Delta_1}{\sqrt{T}}+\frac{D(\theta)\log(1/\delta)}{\sqrt{T}}+\rho((\theta\log(1/\delta))^{2\theta}\sigma^2+G^2)\frac{\log T}{\sqrt{T}}\right)
            \end{align}
        \end{itemize}
        where $D(\theta)$ is defined as in Theorem \ref{thm: SsGD}.
\end{corollary}

Note that the results of Corollary~\ref{coro: time_varying_step_SsGD_subweibull} can reduce to the deterministic case, i.e., $O((\Delta_1+\rho G^2\log T)/\sqrt{T})$, when $\sigma = 0$ for any $\theta > 1$. When $\theta = 1/2$, the sub-Weibull-type noise reduces to sub-Gaussian noise, and our results show that when SsGD minimizes a $\rho$-weakly convex function with such noise, the iteration complexity is at most $O((\Delta_1 + (\log(1/\delta)\sigma^2 + G^2)\log T)/\sqrt{T})$ and we disregard the $\log(1/\delta)$ and $\log T$, our result aligns with the in-expectation results in \citep{Davis_2018}.

To the best of our knowledge, most convergence analysis studies on weakly convex SsGD methods focus on in-expectation results. Therefore, we first compare our findings with the SsGD method in the non-convex setting. 

Both \citep{madden2020high} and \citep{li_2022} provide high-probability upper bounds for the non-convex smooth SGD method (without projection) with sub-Weibull-type gradient noise. For $\theta \geq \frac{1}{2}$, our results share the same order of $T$ and $\log(1/\delta)$ as theirs. However, there is a key difference: \citep{madden2020high}'s work and ours both rely on the assumption of a globally bounded norm for the (sub-)gradients, whereas \citep{lishaojie_liuyong_2022} only requires $\eta_t\|\nabla f(x_t)\| \leq G$ for all $t \in [T]$. The reason for this is that in the descent lemma for $L$-smooth functions, the inner product term can be combined with the square norm of the gradients to form the descent term. In our analysis, however, the objective function is weakly convex, and we use the gradients of the Moreau envelopes to measure convergence. Under weakly convex optimization, it is difficult to remove the $G$-Lipschitz condition (like Assumption~\ref{assp:G_Lipsch}) on $f$.

Note that in Corollary~\ref{coro: time_varying_step_SsGD_subweibull}, we adopt a parameter-free and time-varying step-size. It is natural to consider fixing the time-horizon $T$ and assuming the algorithm has access to problem parameters (e.g., $\Delta_1$, $G$, and $\sigma$) to refine the upper bound in Corollary~\ref{coro: time_varying_step_SsGD_subweibull}.

\begin{corollary}\label{coro: fix_step_SsGD_subweibull}
     Suppose Assumption~\ref{assp:weak_conv}, \ref{assp:lower_bound}, \ref{assp:G_Lipsch}, \ref{assp:unbiased} and \ref{assp:sub_weibull} hold. Let $\{x_t\}$ be the iterate produced by projection-SsGD. Let $f_{1/\bar{\rho}(x_t)}$ is defined as~\eqref{eq: fun_moreau_envelope} and set $\bar{\rho}=3\rho$. Let we fix the time-horizon $T$ and assume that the algorithm can access the parameters $\Delta_1$, $G$, and $\sigma$. Then given any $\delta\in(0,1)$, with the probability at least $1-\delta$
    \begin{itemize}
        \item when $\theta=\frac{1}{2}$, Let $\eta_t\equiv \Theta\left(\sqrt{\frac{\Delta_1}{\rho(\log(1/\delta)\sigma^2+G^2)T}}\right)$ we have following inequality
            \begin{align}
                \frac{1}{T}\sum_{t=1}^T\|\nabla f_{1/3\rho}(x_t)\|^2\leq O\left(\sqrt{\frac{\rho\Delta_1(\log(1/\delta)\sigma^2+G^2)}{T}}+\frac{\log(1/\delta)\sigma^2}{T}\right) 
            \end{align}
        \item when $\theta =(\frac{1}{2},1]$. Let $\eta_t \equiv \Theta\left(\sqrt{\frac{\Delta_1}{\rho(\log(1/\delta)^{2\theta}\sigma^2+G^2)T}}\right)$, we have following inequality 
        \begin{align}
            \frac{1}{T}\sum_{t=1}^T \|\nabla f_{1/3\rho}(x_t)\|^2\leq O&\left(\sqrt{\frac{\rho\Delta_1(\log(1/\delta)^{2\theta}\sigma^2+G^2)}{T}}+\frac{\max\{\sigma^2,G\sigma\}\log(1/\delta)}{T}\right)
        \end{align}
        \item when $\theta>1$ and $\eta_t \equiv  \Theta\left(\sqrt{\frac{\Delta_1}{\rho((\theta\log(1/\delta))^{2\theta}+G^2)}}\right)$, we have following inequality
        \begin{align}
            \frac{1}{T}\sum_{t=1}^T \|\nabla f_{1/3\rho}(x_t)\|^2\leq O&\left(\sqrt{\frac{\rho\Delta_1((\theta\log(1/\delta))^{2\theta}\sigma^2+G^2)}{T}}+\frac{D(\theta)\log(1/\delta)}{T}\right)
        \end{align}
    \end{itemize}
\end{corollary}

When $\sigma = 0$, the results in Corollary~\ref{coro: fix_step_SsGD_subweibull} reduce to the deterministic case $O(\sqrt{(\rho\Delta_1 G^2)/T})$ with $\eta_t \equiv \Theta(\sqrt{\Delta_1/(\rho G^2T)})$. For $\theta = 1/2$, if we ignore the $\log(1/\delta)$ factor, the iteration complexity aligns with the in-expectation result in \citep{Davis_2018}.

Interestingly, when considering $\theta = 1/2$, our result is quite similar to that in \citep{liuzijian_2023}, which analyzes the high probability convergence of  SsGD for convex and $G$-Lipschitz functions with an unbounded constrained domain $\X$ (see Theorem 3.1 in \citep{liuzijian_2023}). 

Notably, even with a constant step-size and fixed time-horizon $T$, the $\log(T/\delta)$ term emerges for $\theta > 1$. As previously discussed, this stems from our use of the sub-Weibull Freedman inequality, perhaps suggesting that the $\log(T/\delta)$ dependency is inherent to sub-Weibull-type gradient noise rather than an artifact of the analysis technique.

\subsubsection{Proof Sketch} 
In this subsection, we present a concise overview of our proof for the weakly convex SsGD method under sub-Weibull-type noise. For conveniently analysis the stochastic process, we denote $\mathcal{F}_t\triangleq \sigma\left(g_1,\dots,g_{t}\right)$, where $\mathcal{F}_t$ is defined on the same probability space $(\Omega,\mathcal{F},P)$. $\E_t[\cdot]$ denotes the conditional expectation given $\mathcal{F}_{t-1}$. We first start with a fundemantal lemma for weakly convex SsGD method.
\begin{lemma}\label{lem: fundemantal_lem}
    Suppose that Assumption~\ref{assp:weak_conv},\ref{assp:lower_bound}, \ref{assp:G_Lipsch} and \ref{assp:unbiased} hold. For any $t\in[T]$ and $\bar{\rho}>\rho$. Let $x_t$ is produced by the projection SsGD method, then we have
    \begin{align}\label{eq:fundemantal_eq}
        \frac{\bar{\rho}-\rho}{\bar{\rho}}\eta_t\|\nabla f_{1/\bar{\rho}}(x_t)\|^2\leq \Delta_t-\Delta_{t+1} + \bar{\rho}\eta_t\langle \hat{x}_t-x_t, \xi_t\rangle + \bar{\rho}\eta_t^2\|\xi_t\|^2 + \bar{\rho}\eta_t^2G^2,
    \end{align}
    where $\Delta_t\triangleq f_{1/\bar{\rho}}(x_t)-\min_{x\in\X} f(x)$, $f(\hat{x}_t) = f_{1/\bar{\rho}}(x_t)$, $\partial_t\triangleq\E_t[g_t]\in\partial f(x_t)$ and $\xi_t\triangleq g_t-\partial_t$.
\end{lemma}
We note that this lemma is analogous to the descent lemma for $L$-smooth functions. However, a key distinction lies in the inner product term: here it is $\langle \hat{x_t} - x_t \rangle$, rather than $\langle x_{t+1} - x_t, \xi_t \rangle$. The similar result as \eqref{eq:fundemantal_eq} can be found  in several literatures analyzing the in-expectation convergence of the SsGD method for weakly convex optimization \citep{Davis_2018,Davis2_2018,Alacaoglu_2020,quanqihu_2024}. Unlike in-expectation convergence analysis, we cannot ignore the martingale $\bar{\rho}\eta_t\langle \hat{x_t}-x_t,\xi_t\rangle$ (i.e., $\E_t[\bar{\rho}\eta_t\langle \hat{x_t}-x_t,\xi_t\rangle]=0$) and must use concentration inequalities to estimate its probability upper bound. Note that \cite{bakhshizadeh_2022} has proven that a sub-Weibull random variable $X$ does not have a moment generating function (MGF) when the tail parameter $\theta>1$, implying that establishing concentration inequalities for the $X$ is challenging. Fortunately, \cite{madden_2024} developed a Bernstein-type concentration inequality (see Theorem 11 in \citep{madden_2024}), which can be regarded as an extension of the Freedman sub-Gaussian inequality. The following lemma can be directly derived from this concentration inequality.
\begin{lemma}\label{lem:sub_weibull_freedman_ineq}
    Given a filtered probability space $(\Omega,\mathcal{F},(\mathcal{F}_t),P)$. Let $\{X_t\}$ and $\{\sigma_t\}$ be adapted to $\mathcal{F}_t\subset \mathcal{F}$, where for any $t\in[T]$, $\E[X_t|\mathcal{F}_{t-1}]=0$ and $\sigma_t$ is nonnegative almost surely. Suppose that $X_t$ is $\sigma_{t-1}$ - sub-Weibull($\theta$) with $\theta\geq\frac{1}{2}$ and $\sigma_{t-1}\leq M_t$ when $\theta>\frac{1}{2}$. Then $\forall \alpha\geq b\max_{t\in[T]}M_t$, the following inequality holds with the probability at least $1-\delta$.
    \begin{align}\label{eq:sub_weibull_freedman_ineq}
    \sum_{t=1}^T X_t \leq 2\alpha\log(\frac{2}{\delta}) + \frac{a}{\alpha}\sum_{t=1}^T \sigma_{t-1}^2, 
\end{align}
where $a$ and $b$ are defined as following:
\begin{align}
    a &= \left\{\begin{aligned}
        &2,\quad \theta = \frac{1}{2}\\
        &(4\theta)^{2\theta}e^2,\quad \theta\in(\frac{1}{2},1]\\
        &(2^{2\theta+1}+2)\Gamma(2\theta+1) + \frac{2^{3\theta}\Gamma(3\theta+1)}{3\log(4T/\delta)^{\theta-1}}, \quad \theta>1
    \end{aligned}\right.\\
    b &= \left\{\begin{aligned}
        &0,\quad \theta = \frac{1}{2}\\
        &(4\theta)^\theta,\quad \theta \in (\frac{1}{2},1]\\
        &2\log(4T/\delta)^{\theta-1},\quad \theta > 1
    \end{aligned}\right.
\end{align}
\end{lemma}
Note that $\bar{\rho}\eta_t\langle x_{t+1} - x_t, \xi_t \rangle$ is $\bar{\rho}\eta_t\|\hat{x_t}-x_t\|\sigma$-sub-Weibull($\theta$). To apply Lemma \ref{lem: fundemantal_lem} to $\bar{\rho}\eta_t\langle x_{t+1} - x_t, \xi_t \rangle$, we additionally require $\bar{\rho}\eta_t\|\hat{x_t}-x_t\|\sigma$ to be upper-bounded by $M_t$ when $\theta>1/2$. Although the constrained domain $\X$ is unbounded, we establish have the following lemma.
\begin{lemma}\label{lem: bound_norm_hx_x}
     Given a $\rho$ - weakly funciton $F$, $\forall \bar{\rho} > \rho$, let $f_{1/\bar{\rho}} (x) $ be a moreau envelope function  of $f(x) $. Then $\forall g \in \partial f(x)$, we have
        \begin{align}~\label{eqn: bound_norm_hx_x}
            \|\hat{x}-x\| \leq \frac{2\|g\|}{\bar{\rho}-\rho}
        \end{align}
\end{lemma}
Combining \ref{lem: bound_norm_hx_x} and Assumption \ref{assp:G_Lipsch}, we can upper-bound $\bar{\rho}\eta_t\|\hat{x_t}-x_t\|\sigma$. This allows us to apply the concentration inequality \eqref{eq:sub_weibull_freedman_ineq} to bound $\bar{\rho}\eta_t\langle x_{t+1} - x_t, \xi_t \rangle$ by $O(\eta_t^2\|\nabla f_{1/\bar{\rho}}(x_t)\|^2)$, which differs from the left-hand side of \eqref{eq:fundemantal_eq} as a descent term. By telescoping the sum, $\Delta_t - \Delta_{t+1}$ can be upper-bounded by $\Delta_1$. Although $\|\xi_t\|^2$ is not a martingale, its summation can be controlled using a appropriate concentration inequalities (see lemma \ref{app_lem: berenstain_type_ineq}).

\subsection{clipped-SsGD under the $p$-BCM bounded noises}\label{sec:Cipped-SsGD_p_BCM}

In this section, we present our high-probability convergence analysis for the weakly convex clipped-SsGD method under $p$-th moment type noise, as specified in Assumption \ref{assp:p_PCM}. The clipping technique is introduced because existing works have shown that the vanilla SsGD method may diverge when $p \in (1, 2)$ under Assumption \ref{assp:p_PCM}.

\begin{algorithm*}[t]
    \caption{Projected SsGD method with clipping}\label{algo:ProjectedClippedSsGD}
    \begin{algorithmic}[1]
       \STATE \textbf{Input}: $x_0 \in \X$, batch size $B>0$, $\eta_0>0$, $\lambda_t > 0$.
       \FOR{$t=1,\cdots,T$}
       \STATE Sample $\mathbf{\zeta}_{1,t},\cdots,\mathbf{\zeta}_{B,t}$ independently from the distribution $\mathcal{D}_\zeta$
       \STATE Access to unbiased stochastic gradient oracle to get $g(x_t,\zeta_{1,t}),\cdots, g(x_t,\zeta_{B,t})$ and compute $\widetilde{g}_t = \frac{1}{B}\sum_{i=1}^B g(x_t,\zeta_{i,t})$.
       \STATE $g_t = \min\left\{\frac{\lambda_t}{\|\tilde{g}_t\|},1\right\}\widetilde{g}_t$.
       \STATE $x_{t+1} = \mathrm{Proj}_\X (x_t - \eta_t g_t)$
       \ENDFOR
    \end{algorithmic}
\end{algorithm*}
Algorithm \ref{algo:ProjectedClippedSsGD} illustrates the update rule for the projected clipped-SsGD method. The sole difference from the standard SsGD method is that we manually constrain the norm of the stochastic gradient to be below the clipping level $\lambda_t$. Now we give a high probabability upper bound of projected SsGD method with clipping for minimizing weakly convex functions with Assumption ~\ref{assp:p_PCM}.
\begin{theorem}\label{thm:projected_ssgd_pcm_HPB_anytime}
    Suppose that Assumptions \ref{assp:weak_conv},\ref{assp:lower_bound}, \ref{assp:G_Lipsch}, \ref{assp:unbiased} and \ref{assp:p_PCM} hold. For any $\lambda,\eta_0\in\mathbb{R}_+$, we let $\lambda_t=\max\{2G,\lambda t^{\frac{1}{p}}\}$ and $\eta_t=\eta_0\min\left\{\frac{1}{\lambda_t}\frac{1}{G\sqrt{t}}\right\}$. Then for  $\bar{\rho}=2\rho$ and any batch-size $B\in\mathbb{N}\setminus\{0\}$, the following inequality holds with the probability at least $1-\delta$
    \begin{align}
    \frac{1}{T}\sum_{t=1}^T\|\nabla f_{1/2\rho}(x_t)\|^2 = 
        O&\left(\left\{\frac{\Delta_1}{\eta_0}+(G+\rho\eta_0)\left((\sigma/\lambda)^pB^{1-p}\log (T)+\log(1/\delta)\right)\right\}\cdot\max\left\{\frac{\lambda}{T^{\frac{p-1}{p}}},\frac{G}{\sqrt{T}}\right\}\right)
    \end{align}
When $\sigma$ is known, let $\lambda = \sigma$ and set $B\equiv1$, we have
\begin{align}
    \frac{1}{T}\sum_{t=1}^T\|\nabla f_{1/2\rho}(x_t)\|^2= O\left(\left\{\frac{\Delta_1}{\eta_0}+(G+\rho\eta_0)\log (T/\delta)\right\}\cdot\max\left\{\frac{\sigma}{T^{(p-1)/p}},\frac{G}{\sqrt{T}}\right\}\right)
\end{align}
where $\Delta_1\triangleq f_{1/2\rho}(x_1)-\min_{x\in\X} f(x)$.
\end{theorem}
Observing the above results, we find that the high-probability convergence rate is $\widetilde{O}(T^{\frac{1-p}{p}})$ in $T$ regardless of whether $\sigma$ is known. When $\sigma$ is known and set to $\sigma=0$ (i.e., the deterministic case), the step-size $\eta_t$ reduces to $\eta_0\min\{1/G\sqrt{t},1/(2G)\}=O(1/\sqrt{t})$, and the clipping level $\lambda_t$ simplifies to $2G$. By Assumption \ref{assp:G_Lipsch}, $\lambda_t$ always exceeds the norm of the sub-gradients of $f$ at any $x$, meaning Algorithm \ref{algo:ProjectedClippedSsGD} reduces to the vanilla projected GD method with a convergence rate of $\tilde{O}(1/\sqrt{T})$ in $T$.
We highlight that the emergence of $\log T$ is solely attributed to the use of a time-varying step-size. Unlike in Corollary $\ref{coro: fix_step_SsGD_subweibull}$, we can eliminate $\log T$ by fixing the time horizon $T$.

\begin{theorem}\label{thm:projected_ssgd_pcm_HPB_fixtime}
    Suppose that Assumptions \ref{assp:weak_conv},\ref{assp:lower_bound}, \ref{assp:G_Lipsch}, \ref{assp:unbiased} and \ref{assp:p_PCM} hold. If we fix the time horizon $T$. For all $\lambda,\eta_0\in\mathbb{R}_+$, we let $\lambda_t\equiv\max\{2G,\lambda T^{\frac{1}{p}}\}$ and $\eta_t=\eta_0\min\left\{\frac{1}{\lambda_t},\frac{1}{G\sqrt{T}}\right\}$. Then for $\bar{\rho}=2\rho$ and any batch-size $B\in\mathbb{N}\setminus\{0\}$, the following inequality holds with the probability at least $1-\delta$
    \begin{align}
    \frac{1}{T}\sum_{t=1}^T\|\nabla f_{1/2\rho}(x_t)\|^2 = 
        O&\left(\left\{\frac{\Delta_1}{\eta_0}+(G+\rho\eta_0)\left((\sigma/\lambda)^pB^{1-p}+\log(1/\delta)\right)\right\}\cdot\max\left\{\frac{\lambda}{T^{\frac{p-1}{p}}},\frac{G}{\sqrt{T}}\right\}\right)
    \end{align}
When $\sigma$ is known, let $\lambda = \sigma$ and set $B\equiv1$, we have
\begin{align}
    \frac{1}{T}\sum_{t=1}^T\|\nabla f_{1/2\rho}(x_t)\|^2=O\left(\left\{\frac{\Delta_1}{\eta_0}+(G+\rho\eta_0)\log(1/\delta)\right\}\cdot\max\left\{\frac{\sigma}{T^{\frac{p-1}{p}}},\frac{G}{\sqrt{T}}\right\}\right)
\end{align}
where $\Delta_1\triangleq f_{1/2\rho}(x_1)-\min_{x\in\X} f(x)$.
\end{theorem}
Note that the dependence on $\delta$ is $\log(1/\delta)$ regardless of whether $\sigma$ is known, provided that $T$ is fixed. This dependence is more optimal than the $\log(T/\delta)$ in the recent work \citep{hong_2025}, which analyzes the high-probability convergence of the AdaGrad-norm method for weakly convex and $G$-Lipschitz functions with sub-Gaussian type gradient noise.

To better make a comparsion with the previous works, we also provide in-expectation version like Theorem \ref{thm:projected_ssgd_pcm_HPB_anytime} and \ref{thm:projected_ssgd_pcm_HPB_fixtime}.
\begin{theorem}\label{thm: clipped_SsGD_inexpectation_any_time}
    Suppose that Assumptions \ref{assp:weak_conv},\ref{assp:lower_bound}, \ref{assp:G_Lipsch}, \ref{assp:unbiased} and \ref{assp:p_PCM} hold. For all $\lambda,\eta_0\in\mathbb{R}_+$, we let $\lambda_t=\max\{2G,\lambda t^{\frac{1}{p}}\}$ and $\eta_t=\eta_0\min\left\{\frac{1}{\lambda_t}\frac{1}{G\sqrt{t}}\right\}$. Then for $\bar{\rho}=2\rho$ and any batch-size $B\in\mathbb{N}\setminus\{0\}$, the following inequality holds
    \begin{align}
        \frac{1}{T}\sum_{t=1}^T \E_t\left[\|\nabla f_{1/2\rho}(x_t)\|^2\right] = O&\left(\left\{\frac{\Delta_1}{\eta_0}+(\rho\eta_0+G)(\sigma/\lambda)^pB^{1-p}\log(eT)+\rho\eta_0\log(eT)\right\}\cdot\max\left\{\frac{\lambda}{T^{\frac{p-1}{p}}},\frac{G}{\sqrt{T}}\right\}\right)
    \end{align}
    If $\sigma$ is known, let $\lambda=\sigma$ and $B=1$, we have
    \begin{align}
         \frac{1}{T}\sum_{t=1}^T \E_t\left[\|\nabla f_{1/2\rho}(x_t)\|^2\right] = O\left(\left\{\frac{\Delta_1}{\eta_0}+(\rho\eta_0+G)\log(eT)\right\}\cdot\max\left\{\frac{\sigma}{T^{\frac{p-1}{p}}},\frac{G}{\sqrt{T}}\right\}\right)
    \end{align}
\end{theorem}
We still can fix the time-horizon $T$ to remove $\log(eT)$.
\begin{theorem}\label{thm:clipped_SsGD_inexpecation_fix_time}
        Suppose that Assumptions \ref{assp:weak_conv},\ref{assp:lower_bound}, \ref{assp:G_Lipsch}, \ref{assp:unbiased} and \ref{assp:p_PCM} hold. Additionally, we fix $T$. For all $\lambda,\eta_0\in\mathbb{R}_+$, we let $\lambda_t\equiv\max\{2G,\lambda T^{\frac{1}{p}}\}$ and $\eta_t\equiv\eta_0\min\left\{\frac{1}{\lambda_t}\frac{1}{G\sqrt{T}}\right\}$. Then for $\bar{\rho}=2\rho$ and any batch-size $B\in\mathbb{N}\setminus\{0\}$, the following inequality holds
    \begin{align}
        \frac{1}{T}\sum_{t=1}^T \E_t\left[\|\nabla f_{1/2\rho}(x_t)\|^2\right] = O&\left(\left\{\frac{\Delta_1}{\eta_0}+(\rho\eta_0+G)(\sigma/\lambda)^pB^{1-p}+\rho\eta_0\right\}\cdot\max\left\{\frac{\lambda}{T^{\frac{p-1}{p}}},\frac{G}{\sqrt{T}}\right\}\right)
    \end{align}
    When $\sigma$ is known, let $\lambda=\sigma$ and $B=1$, we have
    \begin{align}
         \frac{1}{T}\sum_{t=1}^T \E_t\left[\|\nabla f_{1/2\rho}(x_t)\|^2\right] = O\left(\left\{\frac{\Delta_1}{\eta_0}+\rho\eta_0+G\right\}\cdot\max\left\{\frac{\sigma}{T^{\frac{p-1}{p}}},\frac{G}{\sqrt{T}}\right\}\right)
    \end{align}
\end{theorem}
If we analyze the number of sub-gradient evaluations required to find a point $x$ satisfying $\mathbb{E}[\|\nabla f_{1/\bar{\rho}}(x)\|] \leq \varepsilon$, the results in Theorem \ref{thm:clipped_SsGD_inexpecation_fix_time} show that finding such an $x$ requires at most $O\left(\varepsilon^{-\frac{2p}{p-1}}\right)$ sub-gradient evaluations. (Note that since the left-hand side of the inequalities in Theorem \ref{thm:clipped_SsGD_inexpecation_fix_time} is the average squared norm of the gradients of the Moreau envelope functions of $f$, we must use Jensen's inequality, i.e., $\mathbb{E}[\|X\|] \leq \sqrt{\mathbb{E}[\|X\|^2]}$.). \cite{zhang_2020} has shown that the upper bound of the sample complexity for any stochastic algorithm under Assumption \ref{assp:p_PCM} when minimizing $L$-smooth functions cannot be lower than $\Omega(\varepsilon^{-\frac{3p-2}{p-1}})$. Our sample complexity exceeds this lower bound. However, no existing works have provided the lower sample complexity (or conergence rate) of SFOMs under Assumption \ref{assp:p_PCM} for minimizing weakly convex functions. The primary reason is that it remains unclear now whether stochastic gradient methods for non-smooth, non-convex yet weakly convex functions enjoy the same lower sample complexity as those for smooth functions. Considering this, we believe our results still demonstrate value for weakly convex optimization, even though this result may not seem optimal compared to existing works for 
$L$
-smooth optimization. 

Finally, we want to compare with the recent work \citep{hu_2025}, which provides the only in-expectation bound for clipped-SsGD in weakly convex optimization under the distributed stochastic optimization scenario.
\begin{itemize}
    \item \citep{hu_2025} only provides in-expectation convergence guarantees without specifying the convergence rate or complexity (merely a general in-expectation upper bound without concrete choices of $\eta_t$ and $\lambda_t$), making it difficult to derive practical insights. In contrast, we present specific high probability and in-expectation sample complexities that facilitate direct comparison with prior works.
    \item In \citep{hu_2025}, the clipping level $\lambda_t$ cannot adapt to $\sigma$. Their in-expectation convergence guarantee is established under the condition that $\lambda_k$ is increasing and $\lim_{t\rightarrow +\infty}\lambda_t=+\infty$, regardless of whether $\sigma = 0$. In our analysis, $\lambda_t$ is set to $\max\{\sigma t^{1/p},2G\}$ (for the case where $\sigma$ is known; see Theorem \ref{thm:clipped_SsGD_inexpecation_fix_time}), which is also increasing when $\sigma \neq 0$. However, when $\sigma = 0$, $\lambda_t$ is fixed at $2G$ in our framework, meaning the clipping technique has no effect on the stochastic sub-gradients and the clipped-SsGD method reduces to the vanilla GD method.
    \item \citep{hu_2025} gives a example of choices of $\eta_t = 1/(t+1)$ and $\lambda_t = 2G(t+1)^{0.4}$, ensuring the convergence rate attains $O(1/\log T)$. This rate is significantly worse than $O(T^{-\frac{p-1}{p}})$ for any $p \in (1, 2]$ in our analysis.
\end{itemize}
Additionally, \citep{hu_2025} only provides the in-expectation convergence guarantee for the clipped-SsGD method. Inspired by prior works analyzing the high-probability convergence of smooth or convex clipped-SsGD methods—including \citep{Cutkosky_2021,liu_2023c,liuzijian_2023,liu_2023stochasticnonsmoothconvexoptimization,Nguyen_2023,chezhegov_2025}—we can establish high-probability upper bounds for the non-smooth weakly convex clipped-SsGD method under Assumption \ref{assp:p_PCM}. These high-probability bounds are more practical in real-world applications compared to in-expectation upper bounds.

\subsubsection{Proof Sketch} 
In this subsection, for simplicity, we only provide a proof sketch of the high-probability convergence. The definitions of $\mathcal{F}_t$ and $\E_t$ follow the same definitions as the proof sketch in section \ref{sec:SsGD_subweibull}. 
We still start with the Lemma \ref{lem: fundemantal_lem}
\begin{align}\label{eq:clipped_SsGD_baseineq_main_body}
    \frac{\bar{\rho}-\rho}{\bar{\rho}}\eta_t\|\nabla f_{1/\bar{\rho}}(x_t)\|^2\leq \Delta_t-\Delta_{t+1} + \bar{\rho}\eta_t\langle \hat{x}_t-x_t, \xi_t\rangle + \bar{\rho}\eta_t^2\|\xi_t\|^2 + \bar{\rho}\eta_t^2G^2
\end{align}
Note that $\partial_t$ is defined so that $\mathbb{E}_t[\widetilde{g}_t] \in \partial f(x_t)$ and $\xi_t\triangleq g_t-\partial_t$. Unlike the unclipped projected SsGD method, $g_t$ here is the clipped stochastic sub-gradient, meaning $g_t$ is no longer an unbiased estimator of the sub-gradient of $f$ at $x_t$—that is, $\mathbb{E}_t[g_t] \notin \partial f(x_t)$. As a result, the term $\bar{\rho}\eta_t\langle \hat{x}_t - x_t, \xi_t\rangle$ is not a martingale, presenting a challenge to analyze the upper bound of the high probability of the summation of this term.
To address this challenge, we borrow the idea from \citep{Cutkosky_2021} to decompose $\xi_t$ into $\xi_t^u$ and $\xi_t^b$, where $\xi_t^u \triangleq g_t - \mathbb{E}_t[g_t]$ denotes the unbiased part and $\xi_t^b \triangleq \mathbb{E}_t[g_t] - \partial_t$ denotes the biased part. After performing such a decomposition, the inequality \eqref{eq:clipped_SsGD_baseineq_main_body} becomes
\begin{align}\label{eq: decomposition_fundemental_lemma_clipped_SsGD}
    \frac{\bar{\rho}-\rho}{\bar{\rho}}\eta_t\|\nabla f_{1/\bar{\rho}}(x_t)\|^2 &\leq \Delta_t-\Delta_{t+1}+\bar{\rho}\eta_t\langle \hat{x}_t-x_t,\xi_t^u\rangle+\bar{\rho}\eta_t\langle \hat{x}_t-x_t,\xi_t^b\rangle+ \bar{\rho}\eta_t^2(2\|\xi_t^u\|^2+2\|\xi_t^b\|^2+G^2)
\end{align}
To continue the analysis, we need the following key lemma.
\begin{lemma}\label{lem:key_lemma_clipped_SsGD}
    Let $\lambda_t\geq 2G$, then $\forall t\in[T]$, we have
    \begin{align}
    &\E_t[\|\xi_t\|^2],\, \E_t[\|\xi_t^u\|^2], \, \|\xi_t^b\|^2 \leq \frac{10(2-B^{-1})\sigma^p\lambda_t^{2-p}}{B^{p-1}}\nonumber\\
    &\|\xi_t^u\|\leq 2\lambda_t,\,\|\xi_t^b\|\leq \frac{2(2-B^{-1})\sigma^p\lambda_t^{1-p}}{B^{p-1}}
    \end{align}
\end{lemma}
The similar results have appeared in \citep{zhang_2020,gorbunov_2020,zhang_2023,Nguyen_2023,liu_2023c, liu_2023stochasticnonsmoothconvexoptimization,liuzijian_2023,LiuLangqi_2024} with the batch size equal to 1. From Lemma \ref{lem:key_lemma_clipped_SsGD}, we can observe that increasing $\lambda_t$ will lead to an increase in the norm of $\xi_t^u$ while decreasing the norm of $\xi_t^b$. This is because $\|\xi_t^b\| \leq 4\sigma^p\lambda_t^{1-p}/B^{p-1}$ and $1-p \leq 0$ for all $p \in (1, 2]$. This observation inspires us to consider how to choose a proper $\lambda_t$ to balance the orders of $\|\xi_t^u\|$ and $\|\xi_t^b\|$. \citep{liu_2023stochasticnonsmoothconvexoptimization} provides a hint that we can upper bound $\eta_t\lambda_t$, i.e., by setting $\eta_t\lambda_t \leq \eta_0$. This condition also plays a important role when applying concentration inequalities to the martingale terms.

Now after decomposing $\xi$, we can easily observe that the term $\bar{\eta_t}\langle\hat{x}_t-x_t,\xi_t^u\rangle$ is indeed a martingale, which implies that we can use the concentration inequalities to bound the summation again. Actually, we have following lemma
\begin{lemma}\label{lem: clipped_SsGD_HPB1_anytime}
    Let we choose $\lambda_t=\max\{\lambda t^{1/p},2G\}$ and $\eta_t\leq\eta_0/\lambda_t$. We set $\bar{\rho}=2\rho$ the following inequality holds with the probability at least $1-\delta/2$
    \begin{align}
        \sum_{t=1}^T \bar{\rho}\eta_t\langle \hat{x}_t-x_t, \xi_t^u\rangle = O\left(\eta_0G\left(\log(1/\delta)+\sqrt{(\sigma/\lambda)^pB^{1-p}\log(1/\delta)\log T }\right)\right)
    \end{align}
\end{lemma}
The challenge now lies in bounding the summation of $\eta_t^2\|\xi_t\|^2$. Based on Lemma \ref{lem:key_lemma_clipped_SsGD}, we can directly bound it by $4\lambda_t^2$. However, using this bound, we obtain $\eta_t^2\lambda_t^2 \leq \eta_0^2$, which is insufficient to control the summation. We observe that in Lemma \ref{lem:key_lemma_clipped_SsGD}, the bound on $\mathbb{E}_t[\|\xi_t^u\|^2]$ has a lower order than that of $\|\xi_t\|^2$. Therefore, we can decompose $\|\xi_t^u\|^2$ into $\|\xi_t^u\|^2-\E_t[\|\xi_t^u\|^2]+\E_t[\|\xi_t^u\|^2]$. Note that $\|\xi_t^u\|^2 - \mathbb{E}_t[\|\xi_t^u\|^2]$ is a martingale. For the summation of this term, we have the following lemma:
\begin{lemma}\label{lem:clipped_SsGD_HPB2_any_time}
    Let we choose $\lambda_t=\max\{\lambda t^{1/p},2G\}$ and $\eta_t\leq\eta_0/\lambda_t$. We set $\bar{\rho}=2\rho$. The following inequality holds with the probability at least $1-\delta/2$
    \begin{align}
        \sum_{t=1}^T \bar{\rho}\eta_t^2(\|\xi_t^u\|^2 - \mathbb{E}_t[\|\xi_t^u\|^2]) = O\left(\rho\eta_0^2\left(\log(1/\delta)+\sqrt{(\sigma/\lambda)^pB^{p-1}\log(1/\delta)\log T}\right)\right)
    \end{align}
\end{lemma}
From Lemmas \ref{lem: clipped_SsGD_HPB1_anytime} and \ref{lem:clipped_SsGD_HPB2_any_time}, we know that the growth rates of the summations of $\bar{\rho}\eta_t\langle\hat{x}_t - x_t,\xi_t^u\rangle$ and $\bar{\rho}\eta_t^2(\|\xi_t^u\|^2 - \mathbb{E}_t[\|\xi_t^u\|^2])$ can both be bounded by $O(\sqrt{\log T})$ with probability $1 - \delta$. For the term $\bar{\rho}\eta_t\langle\hat{x}_t - x_t, \xi_t^b\rangle$, with the choice $\lambda_t = \max\{\lambda t^{1/p}, 2G\}$ and $\eta_t \leq \frac{\eta_0}{\lambda_t}$, we derive the following via the Cauchy-Schwarz inequality, Lemma \ref{lem: bound_norm_hx_x}, Lemma \ref{lem:key_lemma_clipped_SsGD}, and Assumption \ref{assp:G_Lipsch}:  

\[
\bar{\rho}\eta_t\langle\hat{x}_t - x_t, \xi_t^b\rangle \leq \bar{\rho}\eta_t\|\hat{x}_t - x_t\|\|\xi_t^b\| \leq \frac{2\bar{\rho}G\eta_t\lambda_t^{1-p}\sigma^p}{\bar{\rho} - \rho} \leq \frac{2\bar{\rho}G\eta_0(\lambda/\sigma)^p}{(\bar{\rho} - \rho)t}.
\]  

Thus, the summation of $\bar{\rho}\eta_t\langle\hat{x}_t - x_t, \xi_t^b\rangle$ can be upper-bounded by $O(\log T)$.

For the remaining terms, namely $\bar{\rho}\eta_t^2 \left(2\mathbb{E}_t[\|\xi_t^u\|^2] + 2\|\xi_t^b\|^2 + \|\partial_t\|^2\right)$, their summations can be upper-bounded by $O(\log T)$ using similar techniques. To bound the summation of $\eta_t^2\|\partial_t\|^2$, we additionally impose the condition $\eta_t \leq \frac{\eta_0}{G\sqrt{t}}$.

\section{Conclusion}
This paper focuses on high-probability convergence in stochastic weakly convex optimization under the heavy-tailed sub-gradient noises. We first analyze the high-probability convergence of vanilla projected Stochastic sub-gradient Descent (SsGD) for weakly convex optimization under the sub-Weibull sub-gradient noises. We then assume that sub-gradient noises have a bounded $p$-th central moment ($p\in(1,2]$), analyzing both high-probability and in-expectation convergence of projected SsGD with clipping. Our results provide theoretical guarantees for these methods in non-smooth weakly convex settings under weaker noise assumptions, broadening their applicability to real-world heavy-tailed scenarios.


\bibliographystyle{plainnat} 
\bibliography{ref}

\newpage
\appendix
\section{Technical Lemmas}
We give a basic lemma of weak convexity, which will be instrumental in the subsequent analysis.
\begin{lemma}\label{lem:weakly_convex}
    If a closed function $f$ on a convex set $\X$ is $\rho$-weakly convex, the following statements are equilvent:
    \begin{enumerate}
        \item $f(x)+\frac{\rho}{2}\|x\|^2$ is convex
        \item $f(y)\geq f(x)+\langle g,y-x\rangle-\frac{\rho}{2}\|x-y\|^2,\quad\forall x,y\in\X,\, g\in\partial f(x)$
    \end{enumerate}
\end{lemma}
\textbf{Proof of Lemma \ref{lem: fundemantal_lem}}
We restate Lemma \ref{lem: fundemantal_lem} in a more general form to enable its use in analyzing both projected SsGD and projected SsGD with clipping.
\begin{lemma}\label{app_lem:fundamental_lem}
    Suppose that Assumption~\ref{assp:weak_conv},\ref{assp:lower_bound}, \ref{assp:G_Lipsch} and \ref{assp:unbiased} hold. $\{x_t\}$ is the sequence produced by \ref{algo:Projected_SsGD} or \ref{algo:ProjectedClippedSsGD}. Let $g_t$ be the stochastic sub-gradient in Algorithm \ref{algo:Projected_SsGD} or \ref{algo:ProjectedClippedSsGD} (note that in Algorithm \ref{algo:ProjectedClippedSsGD}, $g_t$ is the clipped sub-gradient) and $\xi_t\triangleq g_t-\partial_t$, where $\partial_t\in\partial f(x_t)$. Then $\forall x_t\in[T],\brho>\rho$, we have
    \begin{align}\label{app_eq:fundemental_lemma}
        \frac{(\bar{\rho}-\rho)\eta_t}{\bar{\rho}}\|\nabla f_{1/\hat{\rho}}(x_t)\|^2 \leq \Delta_t-\Delta_{t+1}+\bar{\rho}\eta_t\left\langle\hat{x}_t-x_t,\xi_t\right\rangle+\bar{\rho}\eta_t^2(\|\xi_t\|^2+G^2),
    \end{align}
    where $\Delta_t\triangleq f_{1/\brho}(x_t)-\min_{x\in\X} f(x)$ and $\hx_t\triangleq \arg\min_{x\in\R^d}\{f(x)+\mathbb{I}_\X(x)+\frac{\brho}{2}\|x-x_t\|^2\}$
\end{lemma}
\begin{proof}
    By the definition of $\hx_{t+1}$, we have
    \begin{align}
        f_{1/\bar{\rho}}(x_{t+1})&=f(\hx_{t+1})+\frac{\brho}{2}\|\hx_{t+1}-x_{t+1}\|^2\nonumber\\
        &\leq f(\hx_t) + \frac{\brho}{2}\|\hx_t-x_{t+1}\|^2\nonumber\\
        &= f(\hx_t) + \frac{\brho}{2}\|\hx_t-\mathrm{Proj}_\X(x_t-\eta_tg_t)\|^2\nonumber\\
        &\overset{(a)}{=}f(\hx_t) + \frac{\brho}{2}\|\mathrm{Proj}_\X(\hx_t)-\mathrm{Proj}_\X(x_t-\eta_tg_t)\|^2\nonumber\\
        &\overset{(b)}{\leq} f(\hx_t) + \frac{\brho}{2}\|\hx_t-x_t+\eta_tg_t\|^2\nonumber\\
        &= f(\hx_t) + \frac{\brho}{2}\|\hx_t-x_t\|^2 + \brho\eta_t\left\langle \hx_t-x_t, g_t \right\rangle + \frac{\brho\eta_t^2}{2}\|g_t\|^2\nonumber\\
        &= f_{1/\brho}(\hx_t) + \brho\eta_t\left\langle \hx_t-x_t, \partial_t \right\rangle + \brho\eta_t\left\langle \hx_t-x_t, \xi_t \right\rangle + \frac{\brho\eta_t^2}{2}\|\xi_t+\partial_t\|^2\nonumber\\
        &\overset{(c)}{\leq} f_{1/\brho}(\hx_t) + \brho\eta_t\left(f(\hx_t)-f(x_t) + \frac{\rho}{2}\|\hx_t-x_t\|^2\right) + \brho\eta_t\left\langle \hx_t-x_t, \xi_t \right\rangle+ \brho\eta_t^2(\|\xi_t\|^2 + G^2)\nonumber\\
        &=f_{1/\brho}(\hx_t) -\brho\eta_t\left[\left(f(x_t)+\frac{\brho}{2}\|x_t-x_t\|^2\right)-\left(f(\hx_t)+\frac{\brho}{2}\|x_t-\hx_t\|^2\right)+\frac{\brho-\rho}{2}\|x_t-\hx_t\|^2\right]\nonumber\\
        &+ \brho\eta_t\left\langle \hx_t-x_t, \xi_t \right\rangle+ \brho\eta_t^2(\|\xi_t\|^2 + G^2)\nonumber\\
        &\overset{(d)}{\leq} f_{1/\brho}(\hx_t) -\brho\eta_t(\brho-\rho)\|\hx_t-x_t\|^2 + \brho\eta_t\left\langle \hx_t-x_t, \xi_t \right\rangle+ \brho\eta_t^2(\|\xi_t\|^2 + G^2)\nonumber\\
        &\overset{(e)}{=} f_{1/\brho}(\hx_t) - \frac{\brho-\rho}{\brho}\eta_t\|\nabla f_{1/\brho}(x_t)\|^2 + \brho\eta_t\left\langle \hx_t-x_t, \xi_t \right\rangle+ \brho\eta_t^2(\|\xi_t\|^2 + G^2)
    \end{align}
    where $(a)$ holds because $\hx_t$ must be in $\X$, $(b)$ holds because the non-expansiveness of the projection, $(c)$ follow the weak convexity of $f$ (see Lemma \ref{lem:weakly_convex}), $(d)$ holds because $\triangleq f(x) + \frac{\brho}{2}\|x-x_t\|^2$ is a $\frac{\brho-\rho}{2}$-strongly convex function. By the definition of strongly convexity, we have
    \begin{align}
        \left(f(x_t)+\frac{\brho}{2}\|x_t-x_t\|^2\right)-\left(f(\hx_t)+\frac{\brho}{2}\|x_t-\hx_t\|^2\right)&\geq \left\langle \hat{\partial}_t + \brho(\hx_t-x_t), x_t-\hx_t\right\rangle+\frac{\brho-\rho}{2}\|x_t-\hx_t\|^2
    \end{align}
    where $\hat{\partial}_t\in\partial f(\hx_t)$. Since $\hx_t$ is the minimizer of $f(x) + \frac{\brho}{2}\|x-x_t\|^2$ over $\X$, the optimality condition implies that the term $\langle \hat{\partial}_t + \brho(\hx_t-x_t), x_t-\hx_t\rangle$ is non-negative. $(e)$ holds because we use $\|\nabla f_{1/\brho}(x_t)\| = \brho\|\hx_t-x_t\|$ Rearranging the terms then completes the proof of Lemma \ref{app_lem:fundamental_lem}.
\end{proof}

\subsection{Proof of Lemma \ref{lem: bound_norm_hx_x}}
\begin{proof}
    Since $\hx = \arg\min_{y\in\R^d}\left\{f(y)+\mathbb{I}_\X(y)+\frac{\brho}{2}\|y-x\|^2\right\}$
    \begin{align}
        f(\hx) + \frac{\brho}{2}\|\hx-x\|^2 \leq f(x) + \frac{\brho}{2}\|x-x\|^2 
    \end{align}
    For any $g\in\partial f(x)$, by the definition of the $\rho$-weakly convexity, we have
    \begin{align}
        f(\hx)-f(x)\geq \langle g, \hx-x\rangle - \frac{\rho}{2}\|\hx-x\|^2
    \end{align}
    Combine these two inequalities, we have
    \begin{align}
        &\frac{\brho-\rho}{2}\|\hx-x\|^2\leq \langle g, x-\hx\rangle\leq \|g\|\|x-\hx\|\nonumber\\
        \Longrightarrow& \|\hx-x\|\leq \frac{2\|g\|}{\brho-\rho}
    \end{align}
\end{proof}

\begin{lemma}{(Theorem 3 in \citep{von_1965inequalities})}\label{app_lem:scalar_batch_size_noise}
Given a probability space $(\Omega,\mathcal{F},P)$, $X_1,X_2,\dots,X_B$ are one-dimiensional random variables defined on this probability space. If $\forall n\in[B]$, $\E[X_n|\sum_{i=1}^{n-1} X_i]=0$ almost surely and $\exists p\in(1,2]$, $\E[|X_n|^p]\leq +\infty$. Then we have
\begin{align}
    \E\left[\left|\sum_{n=1}^B X_n\right|^p\right]\leq (2-B^{-1})\sum_{n=1}^B \E[\left|X_n\right|^p]
\end{align}
    
\end{lemma}
This Lemma can be easily extend to the multi dimensional version
\begin{lemma}\label{app_lem:vector_batch_size_noise}
    Given a probability space $(\Omega,\mathcal{F},P)$, $X_1,X_2,\dots,X_B$ are random vectors defined on this probability space. If $X_1,\dots,X_n$ are i.d.d., and $\exists p\in(1,2]$, $\E[|X_n|^p]\leq +\infty$, we have
    \begin{align}
        \E\left[\left\|\sum_{n=1}^B X_n\right\|^p\right] \leq (2-B^{-1})\sum_{n=1}^B\E[\|X_n\|^p]
    \end{align}
\end{lemma}

\begin{proof}
    We first introduce a standard normal random vector $\beta\sim\mathcal{N}(0,I)$. Then define $z_n$ as $\beta^T X_n$. Note that when $X_n$ is given, $z_n\sim\mathcal{N}(0,\|X_n\|^2)$. Hence, $\E[|z_n|^p|X_n]=\frac{2^{p/2}\Gamma((p+1)/2)}{\sqrt{\pi}}\|X_n\|^p\leq+\infty$. Since $X_1,X_2,\dots,X_n$ are i.d.d., $\E[z_n|\sum_{i=1}^{n-1}z_i]=\E[z_n|X_n]=0$. Then by Lemma \ref{app_lem:scalar_batch_size_noise}, we have
    \begin{align}
        \frac{2^{p/2}\Gamma((p+1)/2)}{\sqrt{\pi}}\left\|\sum_{n=1}^B X_n\right\|^p &= \E_{\beta|X_1,\dots,X_n}\left[\left|\beta^T\sum_{n=1}^B X_n\right|^p\right]\nonumber\\
        & = \E_{\beta|X_1,\dots,X_n}\left[\left|\sum_{n=1}^B z_n\right|^p\right]\nonumber\\
        &\leq (2-B^{-1})\sum_{n=1}^B\E_{\beta|X_n}\left[|z_n|^p\right] \nonumber\\
        &=\frac{2^{p/2}\Gamma((p+1)/2)}{\sqrt{\pi}}(2-B^{-1})\sum_{n=1}^B\|X_n\|^p
    \end{align}
    where we denote $\E_{X|Y}[X]$ as the conditional expectation given $Y$. Now we eliminate $2^{p/2}\Gamma((p+1)/2)/\sqrt{\pi}$ on the both sides and take full expectation then we complete the proof. 
\end{proof}

Notably, both \citep{kornilov_2023acceleratedzerothordermethodnonsmooth,hubler_2025gradientclippingnormalizationheavy} derive a similar result: $\mathbb{E}[\|\sum_{n=1}^B X_n\|^p]\leq 2\sum_{n=1}^B\mathbb{E}[\|X_n\|^p]$. 
When $B=1$, this reduces to $\mathbb{E}[\|X\|^p] \leq 2\mathbb{E}[\|X\|^p]$, indicating that the constant factor preceding $\sum_{n=1}^B\mathbb{E}[\|X_n\|^p]$ is insufficient for the trivial case. Lemma \ref{app_lem:vector_batch_size_noise} provides a modified result that addresses this limitation.

\begin{lemma}\label{app_lem:general_bias_clipped_rand_vec}
    Let $Y_1,\dots,Y_B$ be a sequence of i.d.d random vectors and for each $Y_i$, we have $\E[\|Y_i\|^p]\leq \sigma^p$ and $\E[Y_i]=\mu$. Let $X=\frac{1}{B}\sum_{i=1}^BY_i$ and $\hat{X}=\min\left\{1,\frac{\lambda}{\|X\|}\right\}$. For any $\epsilon>0$, let $\lambda\geq(1+\epsilon)\mu$. Then we have following inequalities 
    \begin{align}
        &\|\hat{X}-\E[\hat{X}]\|\leq 2\lambda,\,\|\E[\hat{X}]-\mu\|\leq \frac{(2-B^{-1})\sigma^p\lambda^{1-p}}{B^{p-1}}(1+\epsilon)\nonumber\\
        &\E[\|\hat{X}-\mu\|^2], \,\E[\|\hat{X}-\E[\hat{X}]\|^2],\,
        \|\E[\hat{X}]-\mu\|^2 \leq \frac{(2-B^{-1})\lambda^{2-p}\sigma^p}{B^{p-1}} \left(\left(1+2/\epsilon\right)^2+1\right)\nonumber
    \end{align}
\end{lemma}

We provide a proof using the technique in prior works \citep{liu_2023stochasticnonsmoothconvexoptimization} and \citep{sadiev2023highprobabilityboundsstochasticoptimization}. Notice that here we slightly relax the condition $\lambda \geq 2G$ as $\lambda\geq(\epsilon+1)G$ to more generalize the results.

\begin{proof}
    Note that 
    \begin{align}
        \E[\norm[p]{X-\mu}] = \E\left[\norm[p]{\frac{1}{B}\sum_{i=2}^B(Y_i-\mu)}\right]&=\frac{1}{B^p}\E\left[\norm[p]{\sum_{i=1}^B(Y_i-\mu)}\right]\nonumber\\
        &\leq \frac{2-B^{-1}}{B^p}\E\left[\sum_{i=1}^B\E[\norm[p]{Y_i-\mu}]\right]\nonumber\\
        &\leq \frac{(2-B^{-1})\sigma^p}{B^{p-1}}
    \end{align}
    where the second inequality holds because we use Lemma \ref{app_lem:vector_batch_size_noise}.

    First, we can easily check
    \begin{align}
        \|\hat{X}-\E[\hat{X}]\|\leq \|\hat{X}\|+\|\E[\hat{X}]\| \leq 2\lambda 
    \end{align}
    Then we analysis $\E[\|\hat{X}-\mu\|^2]$
    \begin{align}
        \E[\|\hat{X}-\mu\|^2] &= \E[\|\hat{X}-\mu\|^21_{\|X\|\geq\lambda}+\|\hat{X}-\mu\|^{2-p}\|\hat{X}-\mu\|^p1_{\|X\|\leq\lambda}]\nonumber\\
        &= \E[\|\hat{X}-\mu\|^21_{\|X\|\geq\lambda}]+ \E[\|\hat{X}-\mu\|^{2-p}\|X-\mu\|^p1_{\|X\|\leq\lambda}]\nonumber\\
        &\overset{(a)}{\leq} \E[\|\hat{X}-\mu\|^21_{\|X-\mu\|\geq(\epsilon\lambda)/(1+\epsilon)}] + (2-B^{-1})\sigma^pB^{1-p}\E[\|\hat{X}-\mu\|^{2-p}]\nonumber\\
        &\overset{(b)}{\leq} \left(\frac{\lambda(2+\epsilon)}{1+\epsilon}\right)^2\E[1_{\|X-\mu\|\geq(\epsilon\lambda)/(1+\epsilon)}] + \left(\frac{\lambda(2+\epsilon)}{1+\epsilon}\right)^{2-p}\frac{(2-B^{-1})\sigma^p}{B^{p-1}}\nonumber\\
        &\overset{(c)}{\leq} \frac{(2-B^{-1})\lambda^{2-p}\sigma^p}{B^{p-1}}\left(\frac{2+\epsilon}{1+\epsilon}\right)^2 \left(\left(\frac{1+\epsilon}{\epsilon}\right)^p+\left(\frac{1+\epsilon}{2+\epsilon}\right)^p\right)
    \end{align}
    where $(a)$ holds because 
    \begin{align}
        &\lambda\leq \|X\| = \|X-\mu+\mu\| \leq \|X-\mu\|+\|\mu\|\leq \|X-\mu\|+\frac{\lambda}{1+\epsilon}\nonumber\\
        &\Longrightarrow \|X-\mu\|\geq \frac{\epsilon\lambda}{1+\epsilon}\Longrightarrow 1_{\|X\|\geq\lambda}\leq 1_{\|X-\mu\|\geq (\epsilon\lambda)/(1+\epsilon)}
    \end{align}
    $(b)$ holds because $\|\hat{X}-\mu\|\leq \|\hat{X}\| + \|\mu\|\leq \lambda + \lambda/(1+\epsilon)$ and $(c)$ holds because by the Markov's inequality, we have
    \begin{align}
        \E[1_{\|X-\mu\|\leq (\epsilon\lambda)/(1+\epsilon)}]
        &= P\{\|X-\mu\|\geq (\epsilon\lambda)/(1+\epsilon)\} \nonumber\\
        &= P\{\|X-\mu\|^p\geq (\epsilon\lambda)^p/(1+\epsilon)^p\}\nonumber\\
        &\leq \frac{\E[\|X-\mu\|^p]}{(\epsilon\lambda)^p/(1+\epsilon)^p} \leq \frac{(2-B^{-1})(1+\epsilon)^p\sigma^p}{\lambda^pB^{p-1}}
    \end{align}
    Note that $\forall \epsilon>0$, the function $g(p,\epsilon)\triangleq ((1+\epsilon)/\epsilon)^p+((1+\epsilon)/(2+\epsilon))^p$ is increasing in $p\in(1,2]$, 
    \begin{align}
        \frac{dg(p,\epsilon)}{dp} &= \left(\frac{1+\epsilon}{\epsilon}\right)^p\log\left(\frac{1+\epsilon}{\epsilon}\right) + \left(\frac{1+\epsilon}{2+\epsilon}\right)^p\log\left(\frac{1+\epsilon}{2+\epsilon}\right)\geq 0\nonumber\\
        &\Longleftrightarrow \left(\frac{2}{\epsilon}+1\right)^p \geq 1\geq \frac{\log\left(\frac{2+\epsilon}{1+\epsilon}\right)}{\log\left(\frac{1+\epsilon}{\epsilon}\right)}
    \end{align}
    Hence let $p=2$, we get
    \begin{align}
        \E[\|\hat{X}-\mu\|^2]\leq \frac{(2-B^{-1})\lambda^{2-p}\sigma^p}{B^{p-1}} \left((1+2/\epsilon)^2+1\right)
    \end{align}
We also notice that
\begin{align}
    \E[\|\hat{X}-\E[\hat{X}]\|^2] &= \E[\|\|\hat{X}-\mu+\mu-\E[\hat{X}]\|^2]\nonumber\\
    &\leq \E[\|\hat{X}-\mu\|^2+\|\mu-\E[\hat{X}]\|^2+2\langle\hat{X}-\mu,\mu-\E[\hat{X}]\rangle]\nonumber\\
    &=\E[\|\hat{X}-\mu\|^2] - \E[\|\mu-\E[\hat{X}]\|^2]\nonumber\\
    &\leq \E[\|\hat{X}-\mu\|^2]\leq
    \frac{(2-B^{-1})\lambda^{2-p}\sigma^p}{B^{p-1}} \left((1+2/\epsilon)^2+1\right)
\end{align}
and 
\begin{align}
    \|\E[\hat{X}-\mu]\|^2\overset{\text{Jensen's ineq}}{\leq} \E[\|\hat{X}-\mu\|^2]\leq \frac{(2-B^{-1})\lambda^{2-p}\sigma^p}{B^{p-1}} \left((1+2/\epsilon)^2+1\right)
\end{align}
Finally, we want to upper bound $\|\E[\hat{X}-\mu]\|$
\begin{align}
    \|\E[\hat{X}]-\mu\| = \|\E[\hat{X}-X]\|&\leq \E[\|\hat{X}-X\|]\nonumber\\
    &= \E[\norm{\min\{1,\lambda/\|X\|\}X-X}]\nonumber\\
    &= \E[\|\min\{0,\lambda/\|X\|-1\}X\|1_{\|X\|\geq\lambda}]\nonumber\\
    &+ \E[\|\min\{0,\lambda/\|X\|-1\}X\|1_{\|X\|\leq\lambda}]\nonumber\\
    &= \E[\|\min\{0,\lambda/\|X\|-1\}X\|1_{\|X\|\geq\lambda}]\nonumber\\
    &= \E[(1-\lambda/\|X\|)\|X\|1_{\|X\|\geq\lambda}]\nonumber\\
    &\leq \E[(\|X\|-\lambda)1_{\|X\|\geq(\epsilon\lambda)/(1+\epsilon)}]\nonumber\\
    &\overset{(a)}{\leq} \E[\|X-\mu\|1_{\|X\|\geq(\epsilon\lambda)/(1+\epsilon)}]\nonumber\\
    &\overset{(b)}{\leq} (\E[\|X-\mu\|^p])^{\frac{1}{p}}(\E[1_{\|X\|\geq(\epsilon\lambda)/(1+\epsilon)}])^{\frac{p-1}{p}}\nonumber\\
    &\leq \frac{(2-B^{-1})^{\frac{1}{p}}\sigma}{B^{(p-1)/p}}(\E[1_{\|X\|\geq(\epsilon\lambda)/(1+\epsilon)}])^{\frac{p-1}{p}}\nonumber\\
    &\leq \frac{(2-B^{-1})\sigma^p\lambda^{1-p}}{B^{p-1}}(1+\epsilon)^{p-1}\nonumber\\
    &\leq \frac{(2-B^{-1})\sigma^p\lambda^{1-p}}{B^{p-1}}(1+\epsilon)
\end{align}
where $(a)$ holds because $\|X\|-\lambda\leq\|X-\E[X]\|+\|\E[X]\|-\lambda \leq \|X-\mu\| - (\epsilon\lambda)/(1+\epsilon)\leq \|X-\mu\| $ and $(b)$ holds because we use the Hölder's inequality.

\end{proof}

\section{Concentration Inequalities}
Here, we list several powerful concentration inequalities to help us to facilitate the analysis of the high-probability convergence of Algorithm \ref{algo:Projected_SsGD} and \ref{algo:ProjectedClippedSsGD}.

\begin{theorem}{(Theorem 11 in \citep{madden_2024})}\label{app_thm:sub_weibull_freedman_ineq}
Given a filtered probability space $(\Omega,\mathcal{F},(\mathcal{F}_t),P)$. Let $\{X_t\}$ and $\{\sigma_t\}$ be adapted to $\mathcal{F}_t\subset \mathcal{F}$, where for any $t\in[T]$, $\E[X_t|\mathcal{F}_{t-1}]=0$ and $\sigma_t$ is nonnegative almost surely. If $X_t$ is $\sigma_{t-1}$ - sub-Weibull($\theta$) with $\theta\geq1/2$, i.e.,
\begin{align}
    \E\left[\exp\left\{(|X_t|/\sigma_{t-1})^{1/\theta}\right\}|\mathcal{F}_{t-1}\right]\leq 2
\end{align}
and we additionally require $\sigma_{t-1}\leq M_t$ when $ \theta>1/2$. Then, for any $x$,$\beta\geq0$ and $\varepsilon\in(0,1)$ we have following inequality.
\begin{align}
    P\left(\bigcup_{\tau\in[T]}\left\{\sum_{t=1}^\tau X_t\geq x\,and\,\sum_{t=1}^\tau a \sigma_{t-1}^2\leq \alpha\sum_{t=1}^\tau X_t+\beta \right\}\right)\leq \exp(-\lambda x+2\lambda^2\beta)+2\varepsilon,
\end{align}
where we require $\alpha\geq b\max_{t\in[T]}M_t$ and $\lambda\in[0,\frac{1}{2\alpha}]$. $a$ and $b$ are defined as following:
\begin{align}
    a &= \left\{\begin{aligned}
        &2,\quad \theta = \frac{1}{2}\\
        &(4\theta)^{2\theta}e^2,\quad \theta\in(\frac{1}{2},1]\\
        &(2^{2\theta+1}+2)\Gamma(2\theta+1) + \frac{2^{3\theta}\Gamma(3\theta+1)}{3\log(T/\varepsilon)^{\theta-1}}, \quad \theta>1
    \end{aligned}\right.\\
    b &= \left\{\begin{aligned}
        &0,\quad \theta = \frac{1}{2}\\
        &(4\theta)^\theta,\quad \theta \in (\frac{1}{2},1]\\
        &2\log(T/\varepsilon)^{\theta-1},\quad \theta > 1
    \end{aligned}\right.
\end{align}
\end{theorem}

\subsection{Proof of Lemma \ref{lem:sub_weibull_freedman_ineq}}

 Lemma \ref{lem:sub_weibull_freedman_ineq} can be directly derived from Theorem \ref{app_thm:sub_weibull_freedman_ineq}.

 \begin{proof}
     Let $\beta=0$
\begin{align}
    P\left(\sum_{t=1}^T X_t \leq x + \frac{a}{\alpha}\sum_{t=1}^T \sigma_{t-1}^2 \right)&= 1-P\left(\sum_{t=1}^T X_t > x + \frac{a}{\alpha}\sum_{t=1}^T \sigma_{t-1}^2\right)\nonumber\\
    &\geq 1 - P\left(\sum_{t=1}^T X_t \geq x\, and \, a\sum_{t=1}^T \sigma_{t-1}^2\leq \alpha\sum_{t=1}^T X_t\right)\nonumber\\
    &\geq 1 - P\left(\bigcup_{\tau\in[T]}\left\{\sum_{t=1}^\tau X_t\geq x\,and\,\sum_{t=1}^\tau a \sigma_{t-1}^2\leq \alpha\sum_{t=1}^\tau X_t \right\}\right)\nonumber\\
    &\geq 1 - \exp(-\lambda x)-2\varepsilon
\end{align}
Let $\lambda = \frac{1}{2\alpha}$, $\varepsilon = \frac{\delta}{4}$ and $x = 2\alpha\log(\frac{2}{\delta})$, we finish the proof.
\end{proof}

\begin{lemma}{(Theorem 1 in \citep{Vladimirova_2020})}\label{app_lem:non_martingale_sub_weibull_freedman_ineq}
    Given a sequence of random variables $X_1,X_2,\dots,X_T$. For any $t\in[T]$, $X_t$ is $\sigma_t$-sub-Weibull($\theta$) ($\theta\geq 1/2$). Then, $\forall\gamma\geq0$, we have
    \begin{align}
        P\left(\left|\sum_{t=1}^T X_t\right|\geq\gamma\right)\leq 2\exp\left\{-\left(\frac{\gamma}{v(\theta)\sum_{t=1}^T\sigma_t}\right)^{1/\theta}\right\}, 
    \end{align}
    where 
    \begin{align}
        v(\theta) = \left\{\begin{aligned}
            &(4e)^\theta,\quad \theta\leq1\\
            &2(2e\theta)^\theta,\quad \theta\geq1
        \end{aligned}\right.
    \end{align}
\end{lemma}

\begin{corollary}\label{app_coro:non_martingale_sub_Weibull_freedman_ineq}
      Suppose that $X_1,X_2,\dots,X_T$ satisfy the conditions in Lemma \ref{app_lem:non_martingale_sub_weibull_freedman_ineq} With the probability at least $1-\delta$, the following inequality holds
     \begin{align}
         |\sum_{t=1}^T X_t|\leq v(\theta)\sum_{t=1}^T \sigma_t\log(\frac{2}{\delta})^\theta,
     \end{align}
     where $v(\theta)$ is as defined in Lemma \ref{app_lem:non_martingale_sub_weibull_freedman_ineq}.
\end{corollary}

This corollary is derived by simply choosing $\gamma = v(\theta)\sum_{t=1}^T\sigma_t\log(\frac{2}{\delta})^\theta$ in Lemma \ref{app_lem:non_martingale_sub_weibull_freedman_ineq}. Notably, neither Lemma \ref{app_lem:non_martingale_sub_weibull_freedman_ineq} nor Corollary \ref{app_coro:non_martingale_sub_Weibull_freedman_ineq} requires $X_t$ to be a martingale.

\begin{lemma}{\citep{freedman1975tail}}\label{app_lem: berenstain_type_ineq}
    Given a filtered probability space ($\Omega$,$\mathcal{F}$,$(\mathcal{F}_t)$,P). Let $\{X_t\}_{t=1}^T$ is a martingale difference sequence adapted to $(\mathcal{F}_t)$. If for any $t\in\mathbb{N}_+$, $|X_t|\leq R$ almost surely and for some certain $T\in\mathbb{N_+}$, $\sum_{t=1}^T \E[|X_t|^2|\mathcal{F}_{t-1}]\leq F$ with the probability $1$. Then for any $\tau\in[T]$, we have following inequality holds with the probability at least $1-\delta$
    \begin{align}
        \sum_{t=1}^\tau\left|X_t \right| \leq \frac{2R\log(2/\delta)}{3}+ \sqrt{2F\log(2/\delta)}
    \end{align}
\end{lemma}

\section{projected SsGD under the $\sigma$-sub-Weibull($\theta$) noises}\label{sec: SsGD_subweibull}

In this section, we formally present the convergence analysis of the projected SsGD method (Algorithm \ref{algo:Projected_SsGD}) under $\sigma$-sub-Weibull($\theta$)-type noise.
\begin{theorem}{(The full version of Theorem \ref{thm: SsGD})}\label{app_thm:full_version_projSsGD_general_thm}
    Suppose Assumption~\ref{assp:weak_conv}, \ref{assp:lower_bound}, \ref{assp:G_Lipsch}, \ref{assp:unbiased} and \ref{assp:sub_weibull} hold. Let $\{x_t\}$ be the iterate produced by projection-SsGD. Let $f_{1/\bar{\rho}(x_t)}$ is defined as~\eqref{eq: fun_moreau_envelope} and set $\bar{\rho}=3\rho$. If $\eta_t$ satisfy 
    \begin{align}
        \sum_{t=1}^T \eta_t = +\infty,\quad \sum_{t=1}^T\eta_t^2 \leq +\infty
    \end{align}
    Then given any $\delta\in(0,1)$, with the probability at least $1-\delta$:
    \begin{itemize}
        \item when $\theta=\frac{1}{2} $, we have following inequality
        \begin{align}
            \sumT\frac{\eta_t}{\sumT\eta_t}\norm[2]{\nab f_{1/3\rho}(x_t)} &\leq \frac{3\Delta_1+36\sigma^2\max_{t\in[t]}\eta_t\log(4/\delta)}{\sumT\eta_t}\nonumber\\
            &+\Brace{98\log(4/\delta)\sigma^2+9G^2}\rho\frac{\sumT\eta_t^2}{\sumT\eta_t}  
        \end{align}
        \item when $\theta\in(\frac{1}{2},1]$, we have following inequality 
        \begin{align}
            \sumT \frac{\eta_t}{\sumT\eta_t}\norm[2]{\nab f_{1/3\rho}}&\leq \frac{3\Delta_1+\max\{2130\sigma^2,72G\sigma\}\log(4/\delta)\max_{t\in[T]}\eta_t}{\sumT\eta_t}\nonumber\\
    &+ (2130\sigma^2\log(4/\delta)^{2\theta}+9G^2)\frac{\sumT\rho\eta_t^2}{\sumT\eta_t}
        \end{align}
        \item when $\theta>1$, we have following inequality
        \begin{align}
            \sumT\frac{\eta_t}{\sumT\eta_t}\norm[2]{\nab f_{1/3\rho}(x_t)}&\leq\frac{3\Delta_1+\Brace{18(11\theta\log(4/\delta))^{2\theta}\sigma^2+9G^2}\sumT\rho\eta_t^2}{\sumT\eta_t}\nonumber\\
    &+ \frac{\widehat{D}(\theta)\log(4/\delta)\max_{t\in[T]}\eta_t}{\sumT\eta_t}
        \end{align}
    \end{itemize}
    where $\widehat{D}(\theta)=\max\left\{15\times2^{3\theta+1}\Gamma(3\theta+1)\sigma^2,36G\sigma\log(8T/\delta)^{\theta-1}\right\}$
\end{theorem}

\begin{proof}

We first take the summation from $1$ to $T$ on both sides of \eqref{app_eq:fundemental_lemma} in Lemma \ref{app_lem:fundamental_lem}.

\begin{align}\label{app_eq:summation_fundemental_lemma}
    \frac{\brho-\rho}{\brho}\sum_{t=1}^T \eta_t\|\nabla f_{1/\brho}(x_t)\|^2 \leq \Delta_1 + \sum_{t=1}^T\brho\eta_t\left\langle \hx_t-x_t,\xi_t\right\rangle + \sum_{t=1}^T \brho \eta_t^2 \|\xi_t\|^2 + \brho G^2\sum_{t=1}^T \eta_t^2
\end{align}
Note that since $\forall x\in\X$, $f_{1/\brho}(x)\geq \min_{x\in\X}f$, we have $\Delta_1-\Delta_{T+1}\leq \Delta_1$. 

We first analysis $\sumT \brho\eta_t^2\|\xi_t\|^2$. By the Assumption \ref{assp:sub_weibull}, we have

\begin{align}
    \E\left[\Exp{\Brace{\frac{\brho\eta_t^2\|\xi_t\|^2}{\brho\eta_t^2\sigma^2}}^{\frac{1}{2\theta}}}\right]=\E\left[\Et{\Exp{\Brace{\frac{\brho\eta_t^2\|\xi_t\|^2}{\brho\eta_t^2\sigma^2}}^{\frac{1}{2\theta}}}}\right]\leq 2
\end{align}
which means $\bar{\rho}\eta_t^2\norm[2]{\xi_t^2}$ is $\eta_t^2\sigma^2\brho$-sub-Weibull($2\theta$). By Lemma \ref{app_coro:non_martingale_sub_Weibull_freedman_ineq}, the following inequality holds with 1-$\delta_1$
\begin{align}
    \sumT \brho\eta_t^2\norm[2]{\xi_t} \leq v(2\theta)\sumT \eta_t^2\sigma^2\brho \log{\frac{2}{\delta_1}}^{2\theta}
\end{align}
Note that we assume that $\theta\geq1/2$, thus $v(2\theta)=2(4e\theta)^{2\theta}$. Then we can get
\begin{align}\label{app_eq:HPB1_projSsGD}
     \sumT \brho\eta_t^2\norm[2]{\xi_t} \leq 2\brho\sigma^2\Brace{4e\theta\log{2/\delta_1}}^{2\theta}\sumT \eta_t^2
\end{align}

Now we estimate the probability upper bound of $\sumT \brho\eta_t\innerdot{\hx_t-x_t}{\xi_t}$. Note that $\Et{\brho\eta_t\innerdot{\hx_t-x_t}{\xi_t}}=0$, i.e., $\brho\eta_t\innerdot{\hx_t-x_t}{\xi_t}$ is the sequence martingales and we let 
\begin{align}
    &\sigma_{t-1} = \eta_t\brho\norm{\hx_t-x_t}\sigma\nonumber\\
    &M_t = \frac{2\brho G\sigma\eta_t}{\brho-\rho}\nonumber
\end{align}

Note that by the Cauchy-Schwarz inequality, $\brho\eta_t\langle \hx_t-x_t, \xi_t\rangle \leq \sigma_{t-1}$ holds almost surely, where $\sigma_t$ is adapted to $\mathcal{F}_t$. Using Lemma \ref{lem: bound_norm_hx_x} and Assumption \ref{assp:G_Lipsch}, we ensure $\sigma_{t-1} \leq M_t$ for all $t \in [T]$. and we have

\begin{align}
    \Et{\Exp{\Brace{\brho\eta_t|\innerdot{\hx_t-x_t}{\xi_t}|/\sigma_{t-1}}^{1/\theta}}} \leq \Et{\Exp{(\|\xi_t\|/\sigma})^{1/\theta}} \leq 2,
\end{align}
which implies that $\brho\eta_t\innerdot{\hx_t-x_t}{\xi_t}$ is $\sigma_{t-1}$-sub-Weibull($\theta$). Now we can apply Lemma \ref{lem:sub_weibull_freedman_ineq} and get the following inequality holds with the probability at least $1-\delta_2$
\begin{align}\label{app_eq:HBP2_projSsGD}
    \sumT \brho\eta_t\innerdot{\hx_t-x_t}{\xi_t}&\leq 2\alpha\log{\frac{2}{\delta_2}}+\frac{a}{\alpha}\sumT \sigma^2\eta_t^2 \brho^2\norm[2]{\hx_t-x_t}\nonumber\\
    &= 2\alpha\log{\frac{2}{\delta_2}}+\frac{a}{\alpha}\sumT \sigma^2\eta_t^2 \norm[2]{\nab \fbrho(x_t)}
\end{align}

Now we plug \eqref{app_eq:HPB1_projSsGD} and \eqref{app_eq:HBP2_projSsGD} into \eqref{app_eq:summation_fundemental_lemma} and get the following inequality holds with the probability at least $1-\delta_1-\delta_2$

\begin{align}
    \sumT \eta_t\Brace{\frac{\brho-\rho}{\brho}-\frac{a\sigma^2}{\alpha}\eta_t}\norm[2]{\nab \fbrho(x_t)} &\leq \Delta_1 + 2\alpha\log(2/\delta_2)\nonumber\\ &+\Brace{2\brho\sigma^2(4e\theta\log(2/\delta_1))^{2\theta}+\brho G^2}\sumT\eta_t^2
\end{align}

Since both $a$ and $\alpha$ are dependent on the value of $\theta$, we discuss the probability upper bounds corresponding to different ranges of $\theta$. For convenience, we first set $\brho = 3\rho$.

\textbf{Case 1. $\theta=\frac{1}{2}$}

When $\theta=1/2$, $a=2$ and we only require $\alpha$ is nonnegative. We Let $\alpha=6\sigma^2\max_{t\in[T]}\eta_t$, then we have
\begin{align}
    \frac{\brho-\rho}{\brho}-\frac{a\sigma^2}{\alpha}\eta_t=\frac{2}{3} - \frac{\eta_t}{\max_{t\in[T]}\eta_t}\geq\frac{1}{3}
\end{align}
Then we have
\begin{align}
    \frac{1}{3}\sumT \eta_t\norm[2]{\nab f_{1/3\rho}(x_t)}&\leq \Delta_1 + 12\sigma^2\max_{t\in[T]}\eta_t\log(2/\delta_2)+ \Brace{12e\log(2/\delta_1)\sigma^2+3G^2}\sumT \rho\eta_t^2 
\end{align}

Let $\delta_1=\delta_2=\delta/2$ and multiply $3\Brace{\sumT\eta_t}^{-1}$ on the both sides we get
\begin{align}
    \sumT\frac{\eta_t}{\sumT\eta_t}\norm[2]{\nab f_{1/3\rho}(x_t)} &\leq \frac{3\Delta_1+36\sigma^2\max_{t\in[t]}\eta_t\log(4/\delta)}{\sumT\eta_t}+ \Brace{36e\log(4/\delta)\sigma^2+9G^2}\frac{\sumT\rho\eta_t^2}{\sumT\eta_t} 
\end{align}

\textbf{Case 2. $\theta\in(\frac{1}{2},1]$}

When $\theta \in(1/2,1]$, $a=(4\theta)^{2\theta}e^2$ and $\alpha$ have to be bigger than $(4\theta)^\theta\max_{t\in[T  ]}M_t=2(4\theta)^\theta\brho G\sigma\max_{t\in[T]}\eta_t/(\brho-\rho)=3(4\theta)^\theta G\sigma\max_{t\in[T]}\eta_t$. 

We Let $\alpha = \max\{3(4\theta)^{2\theta}e^2\sigma^2,3(4\theta)^{\theta}G\sigma\max_{t\in[T]}\eta_t\}$, then we have
\begin{align}
    \frac{\brho-\rho}{\brho}-\frac{a\sigma^2}{\alpha}\eta_t = \frac{2}{3} - \frac{(4\theta)^{2\theta}e^2\sigma^2}{3\max\{(4\theta)^{2\theta}e^2\sigma^2,(4\theta)^{\theta}G\sigma\}}\cdot\frac{\eta_t}{\max_{t\in[T]}\eta_t}\geq \frac{1}{3}
\end{align}

Then we have the following inequality holds with the probability at leat $1-\delta_1-\delta_2$
\begin{align}
    \frac{1}{3}\sumT \eta_t\norm[2]{\nab f_{1/3\rho}(x_t)}&\leq \Delta_1 + 6(4\theta)^{\theta}\max\{(4\theta)^{\theta}e^2\sigma^2,G\sigma\}\log(2/\delta_2)\max_{t\in[T]}\eta_t\nonumber\\
    &+ \Brace{6\sigma^2(4e\theta\log(2/\delta_1))^{2\theta}+3\rho G^2}\sumT\eta_t^2\nonumber\\
    &\leq\Delta_1 + 24\max\{4e^2\sigma^2,G\sigma\}\log(2/\delta_2)\max_{t\in[T]}\eta_t\nonumber\\
    &+ \Brace{96e^2\sigma^2\log(2/\delta_1)^{2\theta}+3G^2}\sumT \rho\eta_t^2\nonumber\\
    &\leq \Delta_1 + \max\left\{710\sigma^2,24\sigma G\right\}\log(2/\delta_2)\max_{t\in[T]}\eta_t\nonumber\\
    &+ \Brace{710\sigma^2\log(2/\delta_1)^{2\theta}+3G^2}\sumT \rho\eta_t^2
\end{align}

where the second inequality holds since for $\theta \in (1/2, 1]$, we have $\theta^{2\theta} \leq \theta^\theta \leq \sqrt{\theta}$.
Let $\delta_1=\delta_2=\delta/2$ and multiply $3\Brace{\sumT\eta_t}^{-1}$ on the both sides we get

\begin{align}
    \sumT \frac{\eta_t}{\sumT\eta_t}\norm[2]{\nab f_{1/3\rho}}&\leq \frac{3\Delta_1+\max\{2130\sigma^2,72G\sigma\}\log(4/\delta)\max_{t\in[T]}\eta_t}{\sumT\eta_t}\nonumber\\
    &+ (2130\sigma^2\log(4/\delta)^{2\theta}+9G^2)\frac{\sumT\rho\eta_t^2}{\sumT\eta_t}
\end{align}

\textbf{Case 3. $\theta>1$}

When $\theta>1$, $a=(2^{2\theta+1}+2)\Gamma(2\theta+1)+\frac{2^{3\theta}\Gamma(3\theta+1)}{(3\log(4T/\delta_2)^{\theta-1})}$ and we require that $\alpha\geq 6G\sigma\log(4T/\delta_2)^{\theta-1}\max_{t\in[T]}\eta_t$. 

We let $\alpha = \max\left\{3(2^{2\theta+1}+2)\Gamma(2\theta+1)\sigma^2+\frac{2^{3\theta}\Gamma(3\theta+1)\sigma^2}{(\log(4T/\delta_2)^{\theta-1})},6G\sigma\log(4T/\delta_2)^{\theta-1}\right\}\max_{t\in[T]}\eta_t$ and also can make sure $\frac{(\brho-\rho)}{\brho} - \frac{(a\sigma^2\eta_t)}{\alpha}\geq1/3$. Then we have following inequality holds with the probability at least $1-\delta_1-\delta_2$.

\begin{align}
    \frac{1}{3}\sumT \eta_t\norm[2]{\nab f_{1/3\rho}(x_t)}&\leq \Delta_1 + \max\left\{6(2^{2\theta}+2)\Gamma(2\theta+1)\sigma^2+\frac{2^{2\theta}\Gamma(3\theta+1)\sigma^2}{\log(4T/\delta_2)^{\theta-1}}, \right.\nonumber\\
    &\left.12G\sigma\log(4T/\delta_2)^{\theta-1}\right\}\log(2/\delta_2)\max_{t\in T}\eta_t\nonumber\\
    &+\Brace{6(4e\theta\log(2/\delta_1))^{2\theta}\sigma^2+3G^2}\sumT \rho\eta_t^2\nonumber\\
    &\leq \Delta_1 + \max\left\{5\times2^{3\theta+1}\Gamma(3\theta+1)\sigma^2,12G\sigma\log(4T/\delta_2)^{\theta-1}\right\}\cdot\log(2/\delta_2)\max_{t\in[T]}\eta_t\nonumber\\
    &+ \Brace{6(11\theta\log(2/\delta_1))^{2\theta}\sigma^2+3G^2}\sumT \rho\eta_t^2
\end{align}

Let $\delta_1=\delta_2=\delta/2$ and multiply $3\Brace{\sumT\eta_t}^{-1}$ on the both sides we get

\begin{align}
    \sumT\frac{\eta_t}{\sumT\eta_t}\norm[2]{\nab f_{1/3\brho}(x_t)}&\leq\frac{3\Delta_1+\widehat{D}(\theta)\log(4/\delta)\max_{t\in[T]}\eta_t}{\sumT\eta_t}\nonumber\\
    &+ \Brace{18(11\theta\log(4/\delta))^{2\theta}\sigma^2+9G^2}\frac{\sumT\rho\eta_t^2  }{\sumT\eta_t}
\end{align}

\end{proof}

\subsection{Proof of Corollary \ref{coro: time_varying_step_SsGD_subweibull}}

\begin{proof}
    If we set $\eta_t = \frac{\gamma}{\sqrt{t}}$, then $\max_{t\in[T]}\eta_t = \gamma$. Similar to the analysis of Theorem \ref{thm: SsGD}, the only difference is that we multiply $2(\eta_T T)^{-1}$ on both sides in the final step of the analysis for each case and note that $\sumT\eta_t^2\leq\gamma^2\log(eT)$. Then we have following results:
    \begin{itemize}
        \item when $\theta=\frac{1}{2}$, we have following inequality
        \begin{align}
            \frac{1}{T}\sumT \norm[2]{\nab f_{1/3\rho}(x_t)} &\leq \frac{3\Delta_1}{\gamma\sqrt{T}} + \frac{36\log(4/\delta)\sigma^2}{\sqrt{T}}+ \gamma\rho(98\log(4/\delta)\sigma^2+9G^2)\frac{\log(eT)}{\sqrt{T}}
        \end{align}
        \item when $\theta=(\frac{1}{2},1]$, we have following inequality
        \begin{align}
            \frac{1}{T}\sumT \norm[2]{\nab f_{1/3\rho}(x_t)}&\leq \frac{3\Delta_1}{\gamma\sqrt{T}} + \frac{\max\{2130\sigma^2,72G\sigma\}\log(4/\delta)}{\sqrt{T}}\nonumber\\
            &+ \gamma\rho(2130\log(4/\delta)^{2\theta}\sigma^2+9G^2)\frac{\log(eT)}{\sqrt{T}}
        \end{align}
        \item when $\theta>1$, we have following inequality 
        \begin{align}
            \frac{1}{T}\sumT \norm[2]{\nab f_{1/3\rho}(x_t)}&\leq \frac{3\Delta_1}{\gamma\sqrt{T}} + \frac{\widehat{D}(\theta)\log(4/\delta)}{\sqrt{T}}+ \gamma\rho(18(11\theta\log(4/\delta))^{2\theta}\sigma^2+9G^2)\frac{\log(eT)}{\sqrt{T}}
        \end{align}
    \end{itemize}
\end{proof}

\subsection{Proof of Corollary \ref{coro: fix_step_SsGD_subweibull}}

\begin{proof}
    Let fix $\eta_t$ as $\eta$, then $\max_{t\in[T]}\eta_t=\eta$. The Theorem \ref{app_thm:full_version_projSsGD_general_thm} becomes

        \begin{itemize}
            \item when $\theta=\frac{1}{2}$, we have
            \begin{align}
                \frac{1}{T}\sumT \norm[2]{\nab f_{1/3\rho}(x_t)}\leq \frac{3\Delta_1}{T\eta}+\rho(98\log(4/\delta)\sigma^2+9G^2)\eta +\frac{36\sigma^2\log(4/\delta)}{T}
            \end{align}
            Then we choose $\eta=\sqrt{\frac{3\Delta_1}{\rho(98\log(4/\delta)\sigma^2+9G^2)T}}$, and we get
            \begin{align}
                \frac{1}{T}\sumT \norm[2]{\nab f_{1/3\rho}(x_t)}\leq\sqrt{\frac{\rho\Delta_1(294\log(4/\delta)\sigma^2+27G^2)}{T}}+\frac{36\log(4/\delta)\sigma^2}{T}
            \end{align}
            \item when $\theta\in(\frac{1}{2},1]$, we have
            \begin{align}
                \frac{1}{T}\sumT \norm[2]{\nab f_{1/3\rho}(x_t)} &\leq \frac{3\Delta_1}{T\eta}
                + (2130\sigma^2\log(4/\delta)^{2\theta}+9G^2)\rho\eta+ \frac{\max\{2130\sigma^2,72G\sigma\}\log(4/\delta)}{T}
            \end{align}
            Let we choose $\eta=\sqrt{\frac{3\Delta_1}{\rho(2130\log(4/\delta)^{2\theta}\sigma^2+9G^2)T}}$, then we get
            \begin{align}
                \frac{1}{T}\sumT \norm[2]{\nab f_{1/3\rho}(x_t)} &\leq \sqrt{\frac{\rho\Delta_1(6390\log(4/\delta)^{2\theta}\sigma^2+27G^2)}{T}}+\frac{\max\{2130\sigma^2,72G\sigma\}\log(4/\delta)}{T}
            \end{align}
            \item when $\theta>1$, we have
            \begin{align}
                \frac{1}{T}\sumT \norm[2]{\nab f_{1/3\rho}(x_t)}&\leq \frac{3\Delta_1}{T\eta}+(18(11\theta\log(4/\delta))^{2\theta}\sigma^2+9G^2)\rho\eta+\frac{\widehat{D}(\theta)\log(4/\delta)}{T}
            \end{align}
            Let $\eta = \sqrt{\frac{3\Delta_1}{\rho(18(11\theta\log(4/\delta))^{2\theta}+9G^2)T}}$, then we get
        \end{itemize}
        \begin{align}
            \frac{1}{T}\sumT \norm[2]{\nab f_{1/3\rho}(x_t)}&\leq \sqrt{\frac{\rho\Delta_1(54(11\theta\log(4/\delta))^{2\theta}\sigma^2+27G^2)}{T}}+\frac{\widehat{D}(\theta)\log(4/\delta)}{T}
        \end{align}
\end{proof}

\section{projected clipped-SsGD under the \textit{p}-BCM noises}
\subsection{Proof of Lemma \ref{lem:key_lemma_clipped_SsGD}}
By the definitions of $\xi_t$, $\xi_t^u$ and $\xi_t^b$, given the filtration $\mathcal{F}_{t-1} = \sigma(g_1, g_2, \dots, g_{t-1})$, $\xi_t$ and $\xi_t^u$ depend only on the $t$-th stochastic gradient oracle, while $\xi_t^b$ is adapted to $\mathcal{F}_{t-1}$. Thus, taking the conditional expectation $\E_t[\cdot] \triangleq \E[\cdot|\mathcal{F}_{t-1}]$, and setting $\epsilon = 1$ and $\lambda_t \geq 2G \geq 2\partial_t$, the results in Lemma \ref{app_lem:general_bias_clipped_rand_vec} remain valid.

Before proving Lemmas \ref{lem: clipped_SsGD_HPB1_anytime} and \ref{lem:clipped_SsGD_HPB2_any_time}, we present the following lemma, which follows directly from our choices of $\eta_t$ and $\lambda_t$.

\begin{lemma}\label{app_lem:eta_lambda_choice_clipped_SsGD}
    If we choose $\lambda_t=\max\{2G,\lambda t^{1/p}\}$ and $\eta_t=\eta_0\min\{1/\lambda_t,1/(G\sqrt{t})\}$, where $\lambda$ and $\eta_0$ can be any positive numbers. We have following inequalities
    \begin{align}
        \eta_t\lambda_t\leq\eta_0,\,(\sigma/\lambda_t)^p\leq (\sigma/\lambda)t^{-1},\, G^2\eta_t^2\leq \eta_0^2t^{-1}
    \end{align}
\end{lemma}

\subsection{Proof of Lemma \ref{lem: clipped_SsGD_HPB1_anytime}}
\begin{proof}
    We first notice that 
    \begin{align}
        |\brho\eta_t\langle\hx_t-x_t,\xi_t^u\rangle|\leq \brho\eta_t\|\hx_t-x_t\|\|\xi_t^u\|\leq \brho\eta_t\cdot\frac{2G}{\brho-\rho}\cdot2\lambda_t\leq 4G\cdot\lambda_t\eta_t\leq 4G\eta_0
    \end{align}
where the second inequality is by the Lemma \ref{lem: bound_norm_hx_x} and \ref{lem:sub_weibull_freedman_ineq} and the last inequality is by Lemma \ref{app_lem:eta_lambda_choice_clipped_SsGD}.
we also have
\begin{align}
    \E_t[|\brho\eta_t\langle\hx_t-x_t,\xi_t^u\rangle|^2]&\leq 4\rho^2\eta_t^2\|\hx_t-x_t\|^2\E[\|\xi_t^u\|^2]\leq 160G^2(\eta_t\lambda_t)^2\cdot\frac{(2-B^{-1})(\sigma/\lambda_t)^p}{B^{p-1}}\nonumber\\
    &\leq 320(G\eta_0)^2B^{1-p}(\sigma/\lambda)^p\cdot\frac{1}{t}
\end{align}
Hence, we can get
\begin{align}
    \sum_{t=1}^T \E_t[|\brho\eta_t\langle\hx_t-x_t,\xi_t^u\rangle|^2] \leq 320(G\eta_0)^2B^{1-p}(\sigma/\lambda)^p\cdot\log(eT)
\end{align}
To apply Lemma \ref{app_lem: berenstain_type_ineq}, we let $R = 4G\eta_0$ and $F=160(G\eta_0)^2B^{1-p}(\sigma/\lambda)^p\cdot\log(eT)$, then we have following inequality holds with the probability at least $1-\delta/2$.
\begin{align}
    \sum_{t=1}^T |\brho\eta_t\langle\hx_t-x_t,\xi_t^u\rangle|&\leq \frac{2R}{3}\log(4/\delta) + \sqrt{2F\log(4/\delta)}\nonumber\\
    &\leq \frac{8\eta_0G\log(4/\delta)}{3} + \sqrt{640 B^{1-p}(G\eta_0)^2(\sigma/\lambda)^p\log(eT)\log(4/\delta)}\nonumber\\
    &\leq 28G\eta_0\left(\log(4/\delta)+\sqrt{B^{1-p}(\sigma/\lambda)^p\log(eT)\log(4/\delta)}\right)
\end{align}

\end{proof}
 
\subsection{Proof of Lemma \ref{lem:clipped_SsGD_HPB2_any_time}}

\begin{proof}
    Firstly, we have
    \begin{align}
        \left|\brho\eta_t^2(\|\xi_t^u\|^2-\E_t\|\xi_t\|^2)\right|\leq \brho\eta_t^2(\|\xi_t^u\|^2+\E_t\|\xi_t\|^2)\leq 8\brho(\lambda_t\eta_t)^2\leq 16\rho\eta_0^2
    \end{align}
    then we get 
    \begin{align}
    \E_t\left[\left|\brho\eta_t^2(\|\xi_t^u\|^2-\E_t\|\xi_t^u\|^2)\right|^2\right] &= \E_t\left[\brho^2\eta_t^4\left(\|\xi_t\|^4-2\langle \|\xi_t\|^2,\E_t\|\xi_t\|^2\rangle+(\E_t\|\xi_t^u\|^2)^2\right)\right]\nonumber\\
    &\leq \brho^2\eta_t^4\left(\E_t\|\xi_t^u\|^4-2(\E_t\|\xi_t^u\|^2)^2+(\E_t\|\xi_t^u\|^2)^2\right)\nonumber\\
    &\leq \brho^2\eta_t^4\E_t\|\xi_t^u\|^4\leq \brho^2\eta_t^4\cdot4\lambda_t^2 \E_t\|\xi_t^u\|^2\nonumber\\
    &\leq 4\brho^2 \eta_0^2 (\eta_t\lambda_t)^2\cdot 20(\sigma/\lambda_t)^pB^{1-p}\nonumber\\
    &\leq 320\rho^2\eta_0^4(\sigma/\lambda)^pB^{1-p}t^{-1}
    \end{align}
    Now, we have the upper bound
    \begin{align}
        \sum_{t=1}^T \E_t\left[\left|\brho\eta_t^2(\|\xi_t^u\|^2-\E_t\|\xi_t^u\|^2)\right|^2\right]\leq 320\rho^2\eta_0^4(\sigma/\lambda)^pB^{1-p}\log(eT)
    \end{align}
    Similarly, we let $R = \brho(\lambda_t\eta_t)^2\leq 16\rho\eta_0^2$ and $F = 320\rho^2\eta_0^4(\sigma/\lambda)^pB^{1-p}\log(eT)$ then apply Lemma \ref{app_lem: berenstain_type_ineq} and get the following inequality holds with the probability at least $1-\delta/2$
    \begin{align}
        \sumT \left|\brho\eta_t^2(\|\xi_t^u\|^2-\E_t\|\xi_t\|^2)\right| &\leq \frac{2R\log(4/\delta)}{3} + \sqrt{2F\log(4/\delta)}\nonumber\\
        &\leq \frac{32}{3}\rho\eta_0^2\log(4/\delta) + \sqrt{640(\sigma/\lambda)^p\log(eT)\log(4/\delta)B^{1-p}}\nonumber\\
        &\leq 26\eta_0^2\rho\left(\log(4/\delta) + \sqrt{(\sigma/\lambda)^p\log(eT)\log(4/\delta)B^{1-p}}\right)
    \end{align}
\end{proof}

\subsection{Proof of Theorem \ref{thm:projected_ssgd_pcm_HPB_anytime}}

\begin{proof}
    Let $\brho=2\rho$ and we take the summation of \eqref{eq: decomposition_fundemental_lemma_clipped_SsGD} from $1$ to $T$, then we get
    \begin{align}\label{eq:summation_decomposition_fundemental_lemma_clipped_SsGD}
       \frac{1}{2}\sumT \eta_t\norm[2]{\nab f_{1/2\rho}(x_t)} &\leq \Delta_1 + \sumT 2\rho\eta_t\langle \hx_t-x_t, \xi_t^u\rangle + \sumT 4\rho\eta_t^2(\|\xi_t^u\|^2-\E_t\|\xi_t^u\|^2)\nonumber\\
        &+ \sumT2\rho\eta_t\langle \hx_t-x_t,\xi_t^b\rangle+ \sumT2\rho\eta_t^2(2\E_t\|\xi_t^u\|^2+2\|\xi_t^b\|^2+G^2)
    \end{align}
    Note that combine Lemma \ref{lem: clipped_SsGD_HPB1_anytime} with Lemma \ref{lem:clipped_SsGD_HPB2_any_time}, we have following inequality holds with the probability at least $1-\delta$.
    \begin{align}\label{eq:summation_HPB_clipped_SsGD}
        &\sumT2\rho\eta_t\langle \hx_t-x_t,\xi_t^b\rangle+ \sumT4\rho\eta_t^2(2\E_t\|\xi_t^u\|^2+2\|\xi_t^b\|^2+G^2)\nonumber\\
        &\leq (28G\eta_0+42\eta_0^2\rho)\left(\log(4/\delta) + \sqrt{(\sigma/\lambda)^p\log(eT)\log(4/\delta)B^{1-p}}\right)\nonumber\\
        &\leq (28G\eta_0+42\eta_0^2\rho)\left(\frac{3}{2}\log(4/\delta)+\frac{(\sigma/\lambda)^p\log(eT)}{2B^{p-1}}\right)
    \end{align}
    For the remind terms, we the have following deterministic following inequalities
    \begin{align}\label{eq:deterministic_summation1_clipped_SsGD}
        \sumT 2\rho\eta_t\langle\hx_t-x_t,\xi_t^b\rangle &\leq \sumT 2\rho\eta_t\|\hx_t-x_t\|\|\xi_t^b\|\leq \sumT 4G\cdot(\eta_t\lambda_t)\cdot \frac{2(2-B^{-1})(\sigma/\lambda_t)^p}{B^{p-1}}\nonumber\\
        &\leq 16G\eta_0(\sigma/\lambda)^pB^{1-p}\sumT t^{-1}\leq 16G\eta_0(\sigma/\lambda)^pB^{1-p}\log(eT)
    \end{align}
    and 
    \begin{align}\label{eq:deterministic_summation2_clipped_SsGD}
        2\rho\sumT\eta_t^2(\E_t\norm[2]{\xi_t^u}+2\norm[2]{\xi_t^b}+G^2)&\leq 2\rho\eta_t^2\left(40(2-B^{-1})\sigma^p\lambda_t^{2-p}B^{1-p}+G^2\right)\nonumber\\
        &\leq \sumT160\rho(\eta_t\lambda_t)^2(\sigma/\lambda_t)^pB^{1-p}+\sumT2\rho G^2\eta_t^2\nonumber\\
        &\leq 160\rho\eta_0^2(\sigma/\lambda)^pB^{1-p}\log(eT) + 2\rho\eta_0^2\log(eT)
    \end{align}
    Then we plug \eqref{eq:summation_HPB_clipped_SsGD}, \eqref{eq:deterministic_summation1_clipped_SsGD} and \eqref{eq:deterministic_summation2_clipped_SsGD} into \eqref{eq:summation_decomposition_fundemental_lemma_clipped_SsGD}, then we get the following inequality holds with the probability at least $1-\delta$
    \begin{align}
        \frac{1}{2}\sumT\eta_t\norm[2]{\nab f_{1/2\rho}(x_t)} &\leq \Delta_1 + \left(30G\eta_0+186\rho\eta_0^2\right)\left((\sigma/\lambda)^pB^{1-p}\log(eT)+\log(eT)\right)\nonumber\\
        &+\left(42G\eta_0+39\eta_0^2\rho\right)\log(4/\delta)+2\rho\eta_0^2\log(eT)\nonumber\\
        &\leq \Delta_1 +\left (42G\eta_0+182\rho\eta_0^2\right)\left((\sigma/\lambda)^pB^{1-p}\log(eT)+\log(4/\delta)\right) \nonumber\\
        &+ 2\rho\eta_0^2\log(eT)
    \end{align}
    Now we multiply $2/(\eta_TT) = (2/\eta_0)\max\{\lambda_T/T,G/\sqrt{T}\}\leq 2/\eta_0\max\{\lambda/T^{(p-1)/p},G/\sqrt{T}\}$ on the both sides of above inequality, then we get
    \begin{align}
        \frac{1}{T}\sumT \norm[2]{\nab f_{1/2\rho}(x_t)} &\leq \left\{\frac{2\Delta_1}{\eta_0}+\left (84G+364\rho\eta_0\right)\left((\sigma/\lambda)^pB^{1-p}\log(eT)+\log(4/\delta)\right)\right.\nonumber\\
        &+ 4\rho\eta_0\log(eT)\Big\}\max\left\{\frac{\lambda}{T^{(p-1)/p}},\frac{G}{\sqrt{T}}\right\}
    \end{align}
    If $\lambda=\sigma$ and $B=1$, then we have
    \begin{align}
        \frac{1}{T}\sumT \norm[2]{\nab f_{1/2\rho}(x_t)} &\leq \left\{\frac{2\Delta_1}{\eta_0}+\left (84G+368\rho\eta_0\right)\log(4eT/\delta)\right\}\cdot\max\left\{\frac{\lambda}{T^{(p-1)/p}},\frac{G}{\sqrt{T}}\right\}
    \end{align}
\end{proof}

\subsection{Proof of Theorem \ref{thm: clipped_SsGD_inexpectation_any_time}}

\begin{proof}
    We take expectation, on the both sides of \eqref{eq:summation_decomposition_fundemental_lemma_clipped_SsGD} then we get
    \begin{align}
        \frac{1}{2}\sum_{t=1}^T \E\left[\norm[2]{\nab f_{1/2\rho}(x_t)}\right] &\leq \Delta_1 + \sumT 2\rho\eta_t\E\langle\hx_t-x_t,\xi_t^b\rangle\nonumber\\
        &+ \sumT 2\rho\eta_t^2 (2\E\|\xi_t^u\|^2+2\E\|\xi_t^b\|^2+G^2)\nonumber\\
        &\leq \Delta_1 + \sumT 2\rho\eta_t\E[\|\hx_t-x_t\|\|\xi_t^b\|]\nonumber\\
        &+ \sumT 2\rho\eta_t^2 \E[(2\E_t\|\xi_t^u\|^2+2\|\xi_t^b\|^2+G^2)]\nonumber\\
        &\leq \Delta_1 + 16G\eta_0(\sigma/\lambda)^pB^{1-p}\log(eT)\nonumber\\
        &+160\rho\eta_0^2(\sigma/\lambda)^pB^{1-p}\log(eT)+ 2\rho\eta_0^2\log(eT)\nonumber\\
        &= \Delta_1 + (16G\eta_0+160\rho\eta_0^2)(\sigma/\lambda)^pB^{1-p}\log(eT)\nonumber\\
        &+2\rho\eta_0^2\log(eT)
    \end{align}
    Multiplying $2/(\eta_TT) = (2/\eta_0)\max\{\lambda_T/T,G/\sqrt{T}\}\leq 2/\eta_0\max\{\lambda/T^{(p-1)/p},G/\sqrt{T}\}$ on the both sides of above inequality then we get
    \begin{align}
        \frac{1}{T} \sumT \E\left[\norm[2]{\nab f_{1/2\rho}(x_t)}\right]&\leq \left\{\frac{2\Delta_1}{\eta_0} + (32G+320\rho\eta_0)(\sigma/\lambda)^pB^{1-p}\log(eT)+4\rho\eta_0\log(eT)\right\}\nonumber\\
        &\cdot \max\left\{\frac{\lambda}{T^{(p-1)/p}},\frac{G}{\sqrt{T}}\right\}
    \end{align}
    If $\lambda=\sigma$ and $B=1$, then we have
    \begin{align}
         \frac{1}{T} \sumT \E\left[\norm[2]{\nab f_{1/2\rho}(x_t)}\right]&\leq\left\{\frac{2\Delta_1}{\eta_0}+\left(32G+324G\rho\eta_0\right)\log(eT)\right\}\cdot\max\left\{\frac{\lambda}{T^{(p-1)/p}},\frac{G}{\sqrt{T}}\right\}
    \end{align}
\end{proof}
\end{document}